\pdfoutput=1
\documentclass{amsart} 
\newif\ifpersonal
\usepackage{amsmath,amscd,amsthm,amssymb,mathrsfs,mathtools,bm} 
\usepackage{stmaryrd}

\usepackage{microtype,lmodern,textcomp} 
\usepackage{enumitem,comment,braket,xspace,tikz-cd} 
\usepackage[utf8]{inputenc} 
\usepackage[T1]{fontenc} 
\usepackage[centering,vscale=0.7,hscale=0.7]{geometry}
\usepackage[hidelinks]{hyperref}
\usepackage[capitalize]{cleveref}
\usepackage{ytableau}
\ytableausetup{mathmode,boxsize=0.3em}

\numberwithin{equation}{section}
\theoremstyle{plain}
\newtheorem{theorem}[equation]{Theorem}
\newtheorem{lemma}[equation]{Lemma}
\newtheorem*{lemma*}{Lemma}

\newtheorem*{claim*}{Claim}

\newtheorem{proposition}[equation]{Proposition}
\newtheorem*{proposition*}{Proposition}
\newtheorem{corollary}[equation]{Corollary}
\newtheorem*{corollary*}{Corollary}
\theoremstyle{definition}
\newtheorem{definition}[equation]{Definition}
\newtheorem{definition-theorem}[equation]{Definition-Theorem}
\newtheorem{definition-lemma}[equation]{Definition-Lemma}
\newtheorem{construction}[equation]{Construction}

\newtheorem{example}[equation]{Example}
\newtheorem{remark}[equation]{Remark}

\theoremstyle{definition}

\numberwithin{equation}{section}

\ifpersonal
\newcommand{\personal}[1]{\textcolor[rgb]{0,0,1}{(Personal: #1)}}

\newcommand{\todo}[1]{\textcolor{red}{(Todo: #1)}}
\else
\newcommand{\personal}[1]{\ignorespaces}
\newcommand{\discussion}[1]{\ignorespaces}
\newcommand{\todo}[1]{\ignorespaces}
\fi

\newcommand{\bbC}{\mathbb C}

\newcommand{\bbG}{\mathbb G}

\newcommand{\bbN}{\mathbb N}

\newcommand{\bbP}{\mathbb P}
\newcommand{\bbQ}{\mathbb Q}

\newcommand{\bbZ}{\mathbb Z}

\newcommand{\cA}{\mathcal A}

\newcommand{\cC}{\mathcal C}

\newcommand{\cF}{\mathcal F}
\newcommand{\cH}{\mathcal H}

\newcommand{\cI}{\mathcal I}

\newcommand{\cL}{\mathcal L}
\newcommand{\cM}{\mathcal M}

\newcommand{\cO}{\mathcal O}

\newcommand{\cW}{\mathcal W}

\newcommand{\bA}{\mathbf A}

\newcommand{\bt}{\mathbf t}
\newcommand{\bs}{\mathbf s}

\makeatletter
\let\save@mathaccent\mathaccent
\newcommand*\if@single[3]{%
	\setbox0\hbox{${\mathaccent"0362{#1}}^H$}%
	\setbox2\hbox{${\mathaccent"0362{\kern0pt#1}}^H$}%
	\ifdim\ht0=\ht2 #3\else #2\fi
}
\newcommand*\rel@kern[1]{\kern#1\dimexpr\macc@kerna}
\newcommand*\widebar[1]{\@ifnextchar^{{\wide@bar{#1}{0}}}{\wide@bar{#1}{1}}}
\newcommand*\wide@bar[2]{\if@single{#1}{\wide@bar@{#1}{#2}{1}}{\wide@bar@{#1}{#2}{2}}}
\newcommand*\wide@bar@[3]{%
	\begingroup
	\def\mathaccent##1##2{%
		\let\mathaccent\save@mathaccent
		\if#32 \let\macc@nucleus\first@char \fi
		\setbox\z@\hbox{$\macc@style{\macc@nucleus}_{}$}%
		\setbox\tw@\hbox{$\macc@style{\macc@nucleus}{}_{}$}%
		\dimen@\wd\tw@
		\advance\dimen@-\wd\z@
		\divide\dimen@ 3
		\@tempdima\wd\tw@
		\advance\@tempdima-\scriptspace
		\divide\@tempdima 10
		\advance\dimen@-\@tempdima
		\ifdim\dimen@>\z@ \dimen@0pt\fi
		\rel@kern{0.6}\kern-\dimen@
		\if#31
		\overline{\rel@kern{-0.6}\kern\dimen@\macc@nucleus\rel@kern{0.4}\kern\dimen@}%
		\advance\dimen@0.4\dimexpr\macc@kerna
		\let\final@kern#2%
		\ifdim\dimen@<\z@ \let\final@kern1\fi
		\if\final@kern1 \kern-\dimen@\fi
		\else
		\overline{\rel@kern{-0.6}\kern\dimen@#1}%
		\fi
	}%
	\macc@depth\@ne
	\let\math@bgroup\@empty \let\math@egroup\macc@set@skewchar
	\mathsurround\z@ \frozen@everymath{\mathgroup\macc@group\relax}%
	\macc@set@skewchar\relax
	\let\mathaccentV\macc@nested@a
	\if#31
	\macc@nested@a\relax111{#1}%
	\else
	\def\gobble@till@marker##1\endmarker{}%
	\futurelet\first@char\gobble@till@marker#1\endmarker
	\ifcat\noexpand\first@char A\else
	\def\first@char{}%
	\fi
	\macc@nested@a\relax111{\first@char}%
	\fi
	\endgroup
}
\makeatother

\newcommand{\tH}{\widetilde H}

\newcommand{\tbeta}{\widetilde\beta}

\newcommand{\tmu}{\widetilde\mu}

\newcommand{\dbb}[1]{[\![#1]\!]}

\newcommand{\blambda}{\boldsymbol{\lambda}}

\newcommand{\pt}{\mathrm{pt}}

\newcommand{\tq}{\tilde q}
\newcommand{\bbk}{\Bbbk}

\newcommand{\id}{\mathrm{id}}
\newcommand{\Id}{\mathrm{Id}}

\newcommand{\Gr}{\mathrm{Gr}}
\newcommand{\iso}{\mathit{iso}}

\newcommand{\vir}{\mathrm{vir}}

\newcommand{\longto}{\longrightarrow}

\newcommand{\tcH}{\widetilde{\cH}}

\newcommand{\tcW}{\widetilde{\cW}}
\newcommand{\tnabla}{\widetilde{\nabla}}

\newcommand{\tbA}{\widetilde{\bA}}

\newcommand{\tf}{\widetilde{f}}
\newcommand{\tphi}{\widetilde{\phi}}

\DeclareMathOperator{\End}{End}
\DeclareMathOperator{\Hom}{Hom}
\DeclareMathOperator{\Aut}{Aut}
\DeclareMathOperator{\Mat}{Mat}
\DeclareMathOperator{\NE}{NE}
\DeclareMathOperator{\GL}{GL}

\DeclareMathOperator{\Spf}{Spf}

\DeclareMathOperator{\Spec}{Spec}

\DeclareMathOperator{\Frac}{Frac}
\DeclareMathOperator{\im}{im} 

\DeclareMathOperator{\PD}{PD} 

\DeclareMathOperator{\Jac}{Jac} 
\DeclareMathOperator{\QH}{QH} 
\DeclareMathOperator{\SG}{SG} 

\DeclareMathOperator{\coker}{coker}

\newcommand{\bigquantum}{\mathrm{big}}

\DeclareMathOperator{\ad}{ad} 

\newcommand{\btau}{\boldsymbol{\tau}}
\newcommand{\by}{\mathbf{y}}

\newcommand{\mir}{\mathrm{mir}}

\DeclareMathOperator{\SL}{SL}

\renewcommand{\top}{\mathrm{top}}

\renewenvironment{abstract}{%
  \quotation
  \small
  \textbf{\textit{\abstractname.}} 
}{\endquotation}

\begin{document}
\title[Unfolding of equivariant F-bundles and applications]{Unfolding of equivariant F-bundles and application to the mirror symmetry of flag varieties}

\author{Thorgal Hinault}
\address{Thorgal Hinault, Department of Mathematics, M/C 253-37, Caltech, 1200 E.\ California Blvd., Pasadena, CA 91125, USA}
\email{thinault@caltech.edu}
\author{Changzheng Li}
\address{School of Mathematics, Sun Yat-sen University, Xingangxi Road 135, Haizhu, Guangzhou, 510275, China, Room 815}
\email{lichangzh@mail.sysu.edu.cn}
\author{Tony Yue YU}
\address{Tony Yue YU, Department of Mathematics, M/C 253-37, Caltech, 1200 E.\ California Blvd., Pasadena, CA 91125, USA}
\email{yuyuetony@gmail.com}
\author{Chi Zhang}
\address{Chi Zhang, Department of Mathematics, M/C 253-37, Caltech, 1200 E.\ California Blvd., Pasadena, CA 91125, USA}
\email{czhang5@caltech.edu}
\author{Shaowu Zhang}
\address{Shaowu Zhang, Department of Mathematics, M/C 253-37, Caltech, 1200 E.\ California Blvd., Pasadena, CA 91125, USA}
\email{szhang7@caltech.edu}
\date{February 7, 2025}


\date{May 14, 2025}

\maketitle

\begin{abstract}
    We establish an unfolding theorem for equivariant F-bundles (a variant of Frobenius manifolds), generalizing Hertling-Manin's universal unfolding of meromorphic connections.
    As an application, we obtain the mirror symmetry theorem for the big quantum cohomology of flag varieties, from the recent works on the small quantum cohomology mirror symmetry, via the equivariant unfolding theorem.
\end{abstract}

\tableofcontents

\addtocontents{toc}{\protect\setcounter{tocdepth}{1}}
\section{Introduction}

\subsection{Motivations}

For a smooth complex projective variety $X$, the Gromov-Witten invariants of $X$ are roughly counts of algebraic curves in $X$ with given genus, class, and constraints (see \cite{Gromov_Pseudoholomorphic_curves,Witten_Two-dimensional_gravity,Kontsevich_Gromov-Witten_classes, Behrend_Intrinsic_normal_cone}).
We can organize the rational (i.e.\ genus zero) Gromov-Witten invariants into a generating series as follows.

Fix a homogeneous basis $(T_i)_{0\leq i\leq N}$ of $H^{\ast} (X,\bbQ )$, and let $(T_{i}^{*})_{0\leq i\leq N}$ denote the dual basis with respect to the Poincar\'e pairing.
Let $\bbQ \dbb{\NE (X,\bbZ )}$ denote the completion of $\bbQ [\NE (X,\bbZ )] = \bbQ [q^{\beta}\, \vert\, \beta\in \NE (X,\bbZ )]$ with respect to the maximal ideal $(q^{\beta} , \,\beta\neq 0)$.

The \emph{genus $0$ Gromov-Witten potential} is a formal power series
\[
  \Phi = \sum_{n\ge 0,\beta} \frac{q^\beta}{n!} \sum_{i_1,\dots,i_n} \braket{T_{i_1}\cdots T_{i_n}}_{0,n}^{\beta} t_{i_1}\cdots t_{i_n} \in \bbQ \dbb{\NE (X,\bbZ )}\dbb{t_0,\dots, t_N} ,
\]
where $\langle \cdots \rangle_{0,n}^{\beta}$ denotes the Gromov-Witten invariants of $X$ of genus $0$, class $\beta$ and cohomological constraints $T_{i_1}, \dots, T_{i_n}$.
It gives rise to the \emph{big quantum cohomology} of $X$, i.e.\ a deformation of the classical cup product on $H^*(X, \bbQ)$:
\begin{align*}
\star \colon H^{\ast} (X,\bbQ ) \otimes H^{\ast} (X,\bbQ ) &\longto H^{\ast} (X,\bbQ ) \otimes \bbQ \dbb{\NE (X,\bbZ )}\dbb{t_0,\dots, t_N} \\
T_i \star T_j &\longmapsto \sum_r \frac{\partial^3\Phi}{\partial t_i \partial t_j \partial t_r} T_{r}^{*} .
\end{align*}

A simpler version called \emph{small quantum cohomology} is the restriction of the big quantum cohomology to $t_i = 0$, for all $i=0, \dots, N$ (or equivalently, by the divisor axiom of Gromov-Witten invariants, for all $i$ with $\deg T_i \neq 2$).
The idea of small quantum cohomology appeared before the big version, first in \cite{Candelas_A_pair_of_Calabi-Yau}, where the small quantum cohomology of a quintic Calabi-Yau threefold was computed using the mirror manifold's periods.
This computation led to the curve counting invariants of the quintic that were previously unknown, and sparked decades of research of enumerative geometry and mirror symmetry from the mathematical viewpoint.

The small quantum cohomology mirror symmetry was proved in various cases, such as complete intersections in projective spaces in \cite{Givental_Equivariant_Gromov-Witten_invariants, Lian_Mirror_principle_I} and toric complete intersections in \cite{Givental_a_mirror_theorem_for_toric_complete_intersections, Reichelt_Sevenheck_non_affine, lly_mirror_principal_II}.
Given that the small quantum cohomology is the restriction of the big quantum cohomology, a natural question is whether mirror symmetry still holds for the big quantum cohomology.
The big quantum cohomology mirror symmetry was proved for projective spaces in \cite{Barannikov_semi_infinite_variation}, for quadric hypersurfaces in \cite{HuXiaowen}, for $\bbP^2$ via tropical geometry in \cite{Gross_Tropical_geometry_and_mirror_symmetry}, for toric varieties in \cite{Iritani_Shift_operators, iritani_a_mirror_construction_for_big, Reichelt_Thomas_Sevenheck_Christian_logarithmic_frobenius_manifolds_hypergeometric_systems_and_quantum_d_modules}, and for toric Deligne-Mumford stacks in \cite{Coates_Hodge-theoretic_mirror_symmetry_for_toric_stacks}.

One tool for such an extension is the reconstruction theorem for Gromov-Witten invariants by Kontsevich-Manin \cite{Kontsevich_Gromov-Witten_classes}, which is the prototype of the universal unfolding of Frobenius manifolds by Hertling-Manin \cite{Hertling_Unfoldings_of_meromorphic_connections} and that of logarithmic Frobenius manifolds by Reichelt \cite{Reichelt_Logarithmic-Frobenius-manifolds}.
This is the essential ingredient in the proof of big quantum cohomology mirror symmetry for projective spaces in \cite{Baranikov_semi_infinite_hodge_structutres,Barannikov_semi_infinite_variation}.
It is shown that the big quantum cohomology can be reconstructed from the small under the condition that the small quantum cohomology (or the classical cohomology) is $H^2$-generated.
The Hertling-Manin unfolding theorem applies more generally to so called (TE)-structures, or F-bundles $(\cH, \nabla)/B$, where $\cH$ is a vector bundle over $B\times \Spf \bbk\dbb{u}$ and $\nabla$ is a flat connection on $\cH$ with poles at $u=0$, such that $\nabla_{u^2\partial_u}$ and $\nabla_{u\xi}$ are regular for any tangent vector field $\xi$ on $B$.
The $H^2$-generation condition is then replaced by two conditions called (IC) and (GC).
For $b\in B$, the residues $\nabla_{u\xi}\vert_{(b,0)}$ and $\nabla_{u^2\partial_u}\vert_{(b,0)}$ are endomorphisms of the fiber $\cH_{b,0}$.
An element $v\in \cH_{b,0}$ satisfies the (GC) condition if the iterated action of these endomorphisms on $v$ generate $\cH_{b,0}$.
It satisfies the (IC) condition if the map $\xi\in T_b B\mapsto \nabla_{u\xi}\vert_{(b,0)} (v)$ is injective. 
Under those two conditions, the F-bundle admits a universal unfolding into a maximal F-bundle. 

Another tool for such an extension from small to big quantum cohomology is the reconstruction from a semisimple point.
In the context of Frobenius manifolds, the structure around a semisimple point was studied in \cite{Dubrovin_Geometry_of_2D,analytic_geometry_of_semisimple_coalescent_frobenius_structures}, and a reconstruction result was proved in \cite{Bayer_Semisimple_exercises,Milanov_Tseng_the_space_of_laurent_polynomials}.
Teleman also studied semisimplicity in the context of topological field theories in \cite{Teleman_2011}.

\medskip
In this paper, we aim to establish the big quantum cohomology mirror symmetry for flag varieties, in the sense of isomorphism of big quantum $D$-modules.
The small quantum cohomology mirror symmetry for general flag varieties was recently established in \cite{chow2024dmodulemirrorconjectureflag}, as an isomorphism of small quantum $D$-modules.
In general the small quantum cohomology of flag varieties is neither $H^2$-generated, nor semisimple, so neither of the above reconstruction methods can be applied here.

The \emph{main discovery} of this paper is that an analogous $H^2$-generation condition can be recovered if we work equivariantly with respect to a torus action.

We first extend the definition of F-bundle (from \cite{KKPY_Birational_invariants,HYZZ_decomposition_and_framing_of_F_bundles}) to equivariant F-bundle (see \cref{def:eq-F-bundle}).
Since the connection $\nabla_{\partial_u}$ is not linear with respect to the equivariant variables, we need to work with infinite rank F-bundles over an infinite dimensional base.
Nevertheless, most of the data can still be reduced to a finite rank (T)-structure relative to $H_T^{\ast} (\pt,\bbQ)$.

Next we extend the (IC) and (GC) conditions to the equivariant setting, and establish an unfolding theorem for equivariant F-bundles under these conditions (see \cref{thm:max-unfolding-arbitrary-rank}).

For application to the mirror symmetry of flag varieties, we produce an unfolding of the $B$-model by constructing an appropriate unfolding of the Landau-Ginzburg superpotential.
We further check the various conditions on the big quantum $D$-module of flag varieties, and apply our equivariant unfolding theorem to obtain the mirror symmetry theorem for the big quantum cohomology of flag varieties.

\subsection{Main results}

An \emph{F-bundle} $(\cH ,\nabla )$ over some base $B$ consists of a vector bundle $\cH$ over $B\times \Spf \bbk\dbb{u}$ and a meromorphic flat connection $\nabla$ with poles at $u=0$, such that $\nabla_{u^2\partial_u}$ and $\nabla_{u\xi}$ are regular for any tangent field $\xi$ on $B$.
If the connection $\nabla$ is only defined in the directions of $B$, we call $(\cH,\nabla )$ a $\bbk$-linear \emph{(T)-structure}.
In order not to create confusion in the infinite rank/dimension setting, we formulate F-bundles and (T)-structures in purely algebraic terms in \cref{sec:eq-F-bundles-and-T-structure}, replacing the vector bundle by a free module, and the connection by derivations.

Let us explain the various conditions involved in our equivariant unfolding theorem.
Let $\bbk$ be a field of characteristic zero, $R$ a $\bbk$-algebra and $(\cH ,\nabla )$ an F-bundle (resp.\ a (T)-structure) over $R\dbb{t_i, i\in I}$ for a countable set $I$, with fiber $H$ at $t=0$, $u=0$.
Residues of $\nabla$ induce $\mathrm{K} \coloneqq \nabla_{u^2\partial_u}\vert_{u=t=0} \in \End_R(H)$,
\[ \mu \colon \bigoplus_{i \in I} R\partial_{t_i} \longrightarrow \End_R (H), \quad \partial_{t_i} \longmapsto \nabla_{u\partial_{t_{i}}}\vert_{u=t=0} ,\]
and for any $v\in H$,
\[ \mu_v \colon \bigoplus_{i \in I} R\partial_{t_i} \longrightarrow H, \quad \partial_{t_i} \longmapsto \nabla_{u\partial_{t_{i}}}\vert_{u=t=0} (v) .\]

The F-bundle $(\cH ,\nabla )$ is called \emph{maximal} if there exists $v \in H$ such that $\mu_v$ is an isomorphism, and $v$ is called a cyclic vector.
We further define the following conditions on $v$ (see \cref{def:IC-GC-conditions})
\begin{itemize}
    \item[(IC)] The map $\mu_v$ is injective.
    \item[(GC)] The orbit of $v$ under the action of the subalgebra $R[\im\mu ,\mathrm{K}]\subset \End_R (H)$ (resp.\ $R[\im \mu]\subset \End_R (H)$ in the case of a (T)-structure) is $H$.
    \item[(GC')] The condition (GC) is satisfied after base change to $\Frac (R)$.
\end{itemize}
The conditions (IC) and (GC) were originally formulated in \cite{Hertling_Unfoldings_of_meromorphic_connections} as necessary conditions to obtain the existence and uniqueness of a maximal unfolding.
We find that when working relative to a ring, condition (GC') is enough for uniqueness, while conditions (IC) and (GC) need to be complemented by the assumption that $\coker\mu_v$ is free in order to construct a maximal unfolding (see \cref{thm:Hertling-Manin-F-bundle-intro} for a precise statement).

\subsubsection{Equivariant unfolding theorem}

For our application to the mirror symmetry of a flag variety $X = G/P$, the F-bundle associated to the quantum cohomology of $X$ does not satisfy conditions (GC) or (GC').
Our new idea is to consider the equivariant quantum cohomology of $X$ induced by the natural torus action.
Note that while the associated (T)-structure is linear over $R \coloneqq H_T^{\ast} (\pt,\bbk)$ and of finite rank, the connection $\nabla_{\partial_{u}}$ in the $u$-direction connection is \emph{not} $R$-linear due to the nontrivial grading on $R$.
Therefore, the associated F-bundle can only be defined over the base field $\bbk$, and hence has infinite rank and depends on infinitely many variables, indexed by a $\bbk$-basis of $H_T^{\ast} (X,\bbk)$.

We introduce the notion of \emph{equivariant F-bundle} in \cref{def:eq-F-bundle}.
Let $I$ be a finite set and choose a $\bbk$-linear basis of a $\bbk$-algebra $R$ indexed by a countable set $K$.
Let $\bt_I = \lbrace t_{i,k} \,|\, (i,k)\in I\times K\rbrace$ denote formal parameters over $\bbk$, and $t_I = \lbrace t_i \,|\, i\in I\rbrace$ formal parameters over $R$.
An equivariant F-bundle consists of the data $\lbrace (\cH ,\nabla ),\allowbreak (\cH_R,\nabla_R), \alpha\rbrace$, where $(\cH ,\nabla )$ is a $\bbk$-linear F-bundle over $\bbk\dbb{\bt_I}$ and $\lbrace (\cH_R,\nabla_R) , \alpha\rbrace$ is an $R$-linear lift over $R\dbb{t_I}$ of the $\bbk$-linear (T)-structure underlying $(\cH, \nabla)$.
An unfolding of an equivariant F-bundle is an extension over a bigger formal base (see \cref{def:unfolding_of_equivariant_F-bundle}).
We also generalize the notion of framing (from \cite[Definition 2.9]{HYZZ_decomposition_and_framing_of_F_bundles}) to equivariant F-bundles (\cref{def:framing-eq-F-bundle}), which consists of framings for the $\bbk$-linear F-bundle and the $R$-linear (T)-structure that are compatible under the lift.

We extend the (IC), (GC), (GC') and maximality conditions to equivariant F-bundles by requiring that the $R$-linear (T)-structure satisfy those conditions.
Our main theorem is the following unfolding theorem for equivariant F-bundles.

\begin{theorem}[Unfolding of equivariant F-bundles, \cref{thm:max-unfolding-arbitrary-rank}] \label{thm: intro}
    Let $\cF = \lbrace (\cH ,\nabla ) , (\cH_R,\nabla_R) ,\alpha \rbrace$ be an equivariant F-bundle over $\bbk\dbb{\bt_I}$, and fix $v\in \cH_R |_{u=t_I=0}$.
    \begin{enumerate}[wide]
        \item If $v$ satisfies (IC), (GC) and $\coker\mu_v$ is free, then $\cF$ admits a maximal unfolding with a cyclic vector induced from $v$.
        \item If $v$ satisfies (GC'), then any two maximal unfoldings of $\cF$ with cyclic vectors induced from $v$ are isomorphic under a unique isomorphism.
    \end{enumerate}
    Furthermore, any framing for $\cF$ induces a unique framing on a maximal unfolding.
\end{theorem}

The fist step in our proof is to establish a formal version of the Hertling-Manin unfolding theorem in the finite rank case (see \cref{thm:Hertling-Manin-F-bundle-intro}).
Then we use it to unfold the $R$-linear (T)-structure.
Finally we conclude by unfolding the $\bbk$-linear F-bundle in the $u$-direction.
The key observation for the last step is the very useful \cref{lemma:lift-T-structure-morphism-to-F-bundle}.
It states that an equivariant F-bundle is uniquely determined by the underlying (T)-structure and the value of the $u$-direction connection at one point, under the assumption that the (T)-structure admits a framing.
This assumption always holds for the $\bbk$-linear (T)-structure associated to an equivariant F-bundle by \cref{lemma:weak-framing-t-directions} and \cref{lifting lemma}.

\begin{proposition}[\cref{lemma:lift-T-structure-morphism-to-F-bundle}]
    For $k=1,2$, let $I_k$ be a countable set and $(\cH_k ,\nabla_k) /R\dbb{t_j ,j\in I_k}$ an F-bundle.
    Let $(f,\Phi)\colon (\cH_1,\nabla_1)_0\rightarrow (\cH_2,\nabla_2)_0$ be a morphism of (T)-structures.
    Assume the (T)-structure $(\cH_1,\nabla_1)_0$ admits a framing.
    Then
    \begin{enumerate}[wide]
        \item $\nabla_1$ is uniquely determined by the underlying (T)-structure and $\nabla_{1,\partial_u}\vert_{t_{I_1} =0}$, and any such data determine a unique F-bundle connection extending the (T)-structure.
        \item $(f,\Phi)$ is an isomorphism of F-bundles if and only if $(f,\Phi)\vert_{t_{I_1}=0}$ is an isomorphism of F-bundles.
    \end{enumerate}
\end{proposition}

Here is our formal version of the Hertling-Manin unfolding theorem we mentioned above.
We also deduce a version for (T)-structures in \cref{thm:hertling-manin-T-structure}.

\begin{theorem}[Formal Hertling-Manin unfolding, \cref{thm:Hertling-Manin-F-bundles}] \label{thm:Hertling-Manin-F-bundle-intro}
    Let $R$ be an integral domain containing $\bbQ$.
    Let $(\cH ,\nabla ) /R\dbb{t_1,\dots, t_d}$ be a finite rank F-bundle.
    Let $v\in \cH/(t_1,\dots, t_d,u)\cH$.
    \begin{enumerate}[wide]
        \item If $v$ satisfies (IC), (GC) and $\coker\mu_v$ is free, then there exists a maximal unfolding with a cyclic vector induced from $v$.
        \item If $v$ satisfies (GC'), then any two maximal unfoldings of $(\cH,\nabla )$ with cyclic vectors induced from $v$ are isomorphic under a unique isomorphism.
    \end{enumerate}
    Furthermore, any framing for $(\cH,\nabla )$ induces a unique framing on a maximal unfolding.
\end{theorem}

Our proof follows mostly the original proof of Hertling and Manin, which was carried out in the complex analytic setting.
In particular, we produce unfoldings using the (GC) condition in \cref{lemma:existence-uniqueness-unfoldings}, which is the formal analogue of \cite[Lemma 2.9]{Hertling_Unfoldings_of_meromorphic_connections}.
While the original proof uses analytic methods to construct a framing of the (T)-structure in which the $u$-direction has a logarithmic pole at $u=1$,
we show that the proof actually works in any framing trivialization.

Since we are working over an integral domain $R$, the (IC) and (GC) conditions are not sufficient to prove existence, and we have to require that $\coker\mu_v$ is free in order to construct a maximal unfolding.
This ensures that we can extend a basis of $\im\mu_v$ to a basis of $\cH /(t_1,\dots, t_d,u)\cH$.
We prove the uniqueness by observing that under (GC'), an unfolding $(\cH',\nabla')$ is characterized by a choice of framing before unfolding and the action of $\nabla'$ on a section that extends $v$.
This allows us to compare unfoldings through their action on a section extending $v$, and to establish the isomorphism. 
A priori, the isomorphism we produce is only defined over $\Frac (R)$, but we note that it is in fact defined over $R$ if the unfoldings are.
A key result is the canonical extension of framing for (T)-structures (\cref{lemma:weak-framing-t-directions}), which was essentially proved in \cite{HYZZ_decomposition_and_framing_of_F_bundles}.

\subsubsection{Application to mirror symmetry of flag varieties}

We apply \cref{thm: intro} to the mirror symmetry of flag varieties $G/P$, where $G$ is a simply-connected complex simple Lie group and $P$ is a parabolic subgroup of $G$.
We begin by reviewing some relevant progress on the mirror symmetry of flag varieties.

On the $A$-side, there was a remarkable presentation of the small quantum cohomology ring $\QH^*(G/P)$ in terms of Peterson variety given in the unpublished lecture notes \cite{Quantum_cohomology_of_G/P} by Peterson.
This was partially verified in \cite{RieJAMS,LamShimo,Cheong}, and was recently proved in \cite{Chow_Peterson} in full generality.
On the $B$-side, Rietsch constructed a mirror Landau-Ginzburg model $(X_P^\vee, \mathcal{W})$ for $G/P$ in \cite{Riet}, and showed the coincidence between  the critical loci of $W$ and the Peterson variety strata. As a consequence, we obtain a first level of small quantum cohomology mirror symmetry in the sense of a ring isomorphism $\QH^*(G/P)\cong {\rm Jac}(\mathcal{W})$.
We refer to \cite{BCFKS, LRYZ} and the references therein for more relevant studies in the special case  $G=\SL(n+1, \mathbb{C})$. 

Furthermore, on the $A$-side, we can consider the Dubrovin connection on the trivial $\QH^*(G/P)$-bundle over $\mathbb{C}^*$, which endows the vector bundle with a quantum $D$-module structure. 
The flag variety $G/P$ admits a natural torus action by the maximal torus $T$ of $G$, so that we can consider the equivariant quantum $D$-module structure as well. 
On the $B$-side, we consider the  Brieskorn lattice $ G_0(X_P^{\vee}, \mathcal{W}, p)$ assoicated to  Rietsch's  equivariant superpotential  mirror to $G/P$ (see \cref{subsec:small-B-model} for more details). The small quantum cohomology mirror symmetry in the sense of isomorphism of small quantum $D$-modules has been studied for certain Grassmannians in \cite{MaRi,PRWMR3534834, PRMR3839374, LaTe}, and was recently established in \cite{chow2024dmodulemirrorconjectureflag}.
In the present paper, we first reformulate this in terms of an isomorphism $\mathcal{F}^A\cong \mathcal{F}^B$ of equivariant F-bundles. 
Then, as an application of \cref{thm: intro}, we obtain the following.

\begin{theorem}[Big quantum $D$-module mirror symmetry, Theorem \ref{thm:big-mirror-symmetry-flag}]\label{thm:intro-big-mirror-symmetry-flag}
    The $A$-model big equivariant F-bundle $\cF^{A,\bigquantum}$ is isomorphic to the $B$-model big equivariant F-bundle $\cF^{B,\bigquantum}$.
    The isomorphism is uniquely determined by the small equivariant quantum $D$-module mirror symmetry. 
\end{theorem}
By taking the non-equivariant limit of the isomorphism in \cref{thm:intro-big-mirror-symmetry-flag}, we obtain a non-equivariant version of the big quantum cohomology mirror symmetry for flag varieties, see \cref{thm:non-equivariant}.

Note that the small quantum cohomology $\QH^*(G/P)$ can be neither semisimple nor $H^2$-generated, such as is the case when $G/P=\SG(2, 2n)$ is the Grassmannian of isotropic planes in Lie type $C_n$ (see \cite{ChPe}).
Therefore, the application of our unfolding theorem is essential in such cases, for which neither the unfolding in \cite{Hertling_Unfoldings_of_meromorphic_connections} nor the semisimple reconstruction in \cite{Teleman_2011} can be applied.

In addition to the mirror statement above, we further anticipate the complex analytic convergence of the mirror map, as well as the compatibility with the pairings on the F-bundles.
These aspects present promising directions for future research.

\subsection*{Acknowledgements}
We would like to thank Giordano Cotti, Mark Gross, Xiaowen Hu, Hiroshi Iritani, Ludmil Katzarkov, Maxim Kontsevich, Leonardo Mihalcea, Todor Milanov, Tony Pantev, Constantin Teleman and Yukinobu Toda for valuable discussions. 
C.\ Li was supported by the National Key R \& D  Program of China No. 2023YFA1009801.
The other authors were partially supported by NSF grants DMS-2302095 and DMS-2245099.

\addtocontents{toc}{\protect\setcounter{tocdepth}{2}}

\section{(T)-structures and equivariant F-bundles} \label{sec:eq-F-bundles-and-T-structure}

We fix a field $\bbk$ of characteristic zero, a $\bbk$-algebra $R$ and a $\bbk$-linear basis $(\lambda_k )_{k\in K}$ of $R$, with $K$ a countable set.

\subsection{Completions} \label{subsec:completions}
We set the conventions for completions of rings of polynomials in infinitely many variables, following \cite[\S 2.1]{iritani_a_mirror_construction_for_big}.
Our reference for topological algebra is \cite[0\S7]{Fujiwara_Foundations-rigid-geometry-I}.

Let $I$ be a countable set indexing indeterminates $t_I = \lbrace t_i , i\in I\rbrace$. 
We denote by $\bbN^{(I)}$ the set of almost zero integer sequences indexed by $I$.
Let $M$ be a module or ring, we denote by $M\dbb{t_I}$ the module consisting of formal power series $\sum_{\alpha\in\bbN^{(I)}} a_{\alpha} t_I^{\alpha}$, where $t_I^{\alpha} \coloneqq\prod_{i\in I} t_i^{\alpha_i}$ and $a_{\alpha}\in M$.
It is the projective limit of the modules $M\dbb{t_i ,i\in I'}$, where $I'\subset I$ runs through finite subsets.
For two countable sets $I$ and $I''$, we have $M\dbb{t_i,i\in I}\dbb{t_i, i\in I''} \simeq M\dbb{t_i,i\in I\cup I''}$.

If $M$ is linearly topologized by the descending chain of submodules $\lbrace M_{\lambda}\rbrace_{\lambda\in \Lambda}$, we equip $M\dbb{t_I}$ with the linear topology induced by the submodules
\begin{equation}\label{eq:topology-power-series}
    M\dbb{t_I}_{\lambda, \cI} \coloneqq \left\lbrace \sum_{\alpha\in\bbN^{(I)}} a_{\alpha} t_I^{\alpha} ,\; a_{\alpha}\in M_{\lambda} \text{ for all } \alpha\in \cI \right\rbrace ,
\end{equation}
where $\lambda\in \Lambda$ and $\cI\subset \bbN^{(I)}$ is a finite set of exponents.
The convergence of a sequence for this topology means that the sequence of coefficients of each monomial converges in $M$.
Hence, if $M$ is complete, so is $M\dbb{t_I}$.
If $R$ is a topological ring and $M$ is a topological $R$-module, then $R\dbb{t_I}$ is a topological ring and $M\dbb{t_I}$ is a topological $R\dbb{t_I}$-module.
If $R$ is a topological $\bbk$-algebra, then $R\dbb{t_I}$ is a topological $\bbk\dbb{t_I}$-algebra. 

Let $R$ be a discrete ring, let $M$ be an $R$-module.
The closure of the monomial ideal $(t_i ,i\in I)\subset R\dbb{t_I}$ is the ideal $\mathcal{J} \coloneqq \lbrace f\in R\dbb{t_I} , \; f\vert_{t_I=0} = 0\rbrace$.
If $I$ is finite, those two ideals coincide and the topology on $M\dbb{t_I}$ is equivalent to the usual $(t_i, i\in I)$-adic topology.
When $I$ is infinite, the $\mathcal{J} $-adic topology is finer, which means that for any finite subset $\cI\subset \bbN^{(I)}$ there exists $n\in\bbN$ such that $\mathcal{J} ^nM\dbb{t_I}\subset M\dbb{t_I}_{\cI}$.

\begin{remark} \label{remark:free-module-power-series} 
    Let $I$ be a countable set, $t_I = \lbrace t_i,i\in I\rbrace$ a set of indeterminates.
    Let $R$ be a topological ring.
    Here are a few facts we will use about modules over $R\dbb{t_I}$.
    \begin{enumerate}[wide]
        \item If $M$ is a free $R$-module, then $M\dbb{t_I}$ is free, and we have a canonical isomorphism $M\otimes_R R\dbb{t_I}\simeq M\dbb{t_I}$ given by $m\otimes 1\mapsto m$.
        \item If $M$ and $M'$ are free $R$-modules, there is a canonical isomorphism of $R\dbb{t_I}$-modules
        \[\Hom_{R\dbb{t_I}} (M\dbb{t_I} , M'\dbb{t_I} ) \simeq \Hom_R (M,M')\dbb{t_I}.\]
        \item If $R$ is discrete and $M$ is a free $R$-module, an element $\Phi\in \End_{R\dbb{t_I}} (M\dbb{t_I})$ is an isomorphism if and only if $\Phi\vert_{t_I =0}\in \End_R (M)$ is an isomorphism.
    \end{enumerate}
    For (3), we may reduce to the case $\Phi = 1 + f$ with $f\in \mathcal{J}  \End_R (M)\dbb{t_I}$.
    Then it suffices to prove that the sequence $\Psi_n\coloneqq \sum_{k=0}^n (-1)^k f^k$ converges in $\End_R (M)\dbb{t_I}$.
    For $m\geq n$ we have $\Psi_{m}- \Psi_n = \sum_{k=n+1}^{m} (-1)^{k} f^k\in \mathcal{J} ^{n+1}\End_R (M)\dbb{t_I}$.
    Since the $\mathcal{J} $-adic topology is finer than the topology \eqref{eq:topology-power-series}, the sequence $(\Psi_n)_n$ is a Cauchy sequence.
    Since $\End_R (M)$ is a discrete space, it is complete.
    We conclude that $(\Psi_n)_n$ converges to an element $\Psi$ such that $\Phi\circ\Psi = \Psi\circ\Phi = 1$.
\end{remark}

\subsection{F-bundles and (T)-structures}

We equip $R$ with the discrete topology.
Given a countable set $I$, the derivations $\partial_{t_j}\colon R\dbb{t_i,i\in I}\rightarrow R\dbb{t_i,i\in I}$ are continuous and linearly independent.
Hence, it makes sense to define a (partial) connection in the $t$-directions on a $R\dbb{t_i ,i\in I}$-module $\cH$ by specifying its action on $\partial_{t_j}$ for all $j\in I$.

\begin{definition}[F-bundle, (T)-structure]
    Let $I$ be a countable set and $t_I \coloneqq \lbrace t_i ,i\in I\rbrace$.
    \begin{enumerate}[wide]
        \item An ($R$-linear) \emph{F-bundle} $(\cH ,\nabla )$ over $R\dbb{t_I}$ is a free $R\dbb{t_I,u}$-module $\cH$ together with a ($R$-linear) connection 
        \begin{align*}
            \nabla_{\partial_{t_i}}&\colon \cH\rightarrow u^{-1}\cH ,\\
            \nabla_{\partial_u} &\colon \cH\rightarrow u^{-2}\cH
        \end{align*}
        satisfying the flatness condition.
        \item An ($R$-linear) \emph{(T)-structure} $(\cH,\nabla )$ over $R\dbb{t_I}$ is a free $R\dbb{t_I,u}$-module $\cH$ together with a ($R$-linear) connection in the $t$-directions
        \[\nabla_{\partial_{t_i}}\colon \cH\rightarrow u^{-1}\cH\]
        satisfying the flatness condition.
    \end{enumerate}
    A \emph{morphism of F-bundles (resp.\ (T)-structures)} $(f,\phi )\colon (\cH,\nabla ) /R\dbb{t_I}\rightarrow (\cH',\nabla') /R\dbb{t_J}$ consists of a continuous map of $R$-algebras $f\colon R\dbb{t_J}\rightarrow R\dbb{t_I}$, and a continuous map of $R\dbb{t_I,u}$-modules $\phi\colon \cH\rightarrow f^{\ast}\cH'\coloneqq \cH'\otimes_{R\dbb{t_J,u}}R\dbb{t_I,u}$ such that $\phi\circ\nabla = f^{\ast}\nabla'\circ\phi$.
    
\end{definition}
Underlying an F-bundle $(\cH ,\nabla )$ over $R\dbb{t_I}$ is an $R$-linear (T)-structure $(\cH ,\nabla )_0$ over $R\dbb{t_I}$ obtained by forgetting $\nabla_{\partial_u}$.
This defines a functor $(\cdot )_0$ from $R$-linear F-bundles to $R$-linear (T)-structures.

Let $(\cH ,\nabla )$ be an $R$-linear (T)-structure over $R\dbb{t_I}$.
A \emph{trivialization} of the (T)-structure is a choice of isomorphism $\cH \simeq H\otimes_R R\dbb{t_I,u}$, where $H$ is a free $R$-module (necessarily isomorphic to $\cH /J\cH$, where $J$ is the closure of the ideal $(t_I,u)$).
Under such an isomorphism, the connection $\nabla$ decomposes as $\nabla_{\partial_{t_i}} = \partial_{t_i} + u^{-1} \bA_i (t_I,u)$,  with $\bA_i\in \End_R (H)\dbb{t_I,u}$. 
We refer to $\bA_i$ as the \emph{connection matrix in the direction $t_i$}.  
Different choices of trivialization produce connection matrices related by the usual gauge-transformation formula.

We introduce special trivializations called framings.

\begin{definition}[Framing] \label{def:framing}
\begin{enumerate}[wide]
    \item A \emph{framing} for an $R$-linear F-bundle (resp.\ an $R$-linear (T)-structure) $(\cH,\nabla)/R\dbb{t_I}$ is a trivialization in which the connection matrices only have negative powers of $u$.
    \item A morphism of framed F-bundles (resp.\ (T)-structures) \[(f,\phi)\colon (\cH,\nabla)/R\dbb{t_I}\rightarrow (\cH',\nabla')/R\dbb{t_J}\] is \emph{compatible with the framings} if it is constant when read in framing trivializations. 
    More precisely, the framings $\cH\simeq H\otimes_R R\dbb{t_I,u}$ and $\cH'\simeq H'\otimes_R R\dbb{t_J,u}$ induce an isomorphism
    \[\Hom_{R\dbb{t_I,u}} (\cH, f^{\ast}\cH') \simeq \Hom_R (H,H')\dbb{t_I,u}. \]
    The condition is that the image of $\phi$ is independent of $t_I$ and $u$.
\end{enumerate}
\end{definition}

\subsection{Lift of (T)-structures} \label{subsec:lift-T-structure}
Recall that we have fixed $R$ a $\bbk$-algebra and a $\bbk$-basis $\blambda = (\lambda_k ,k\in K)$ of $R$.
Let $I$ be a countable set, we introduce two sets of formal variables 
\begin{align*}
    t_I\coloneqq \lbrace t_i ,i\in I\rbrace , \quad \bt_I\coloneqq \lbrace t_{i,k} ,(i,k)\in I\times K\rbrace .
\end{align*}
There is a continuous morphism of $R$-algebras
\begin{align} \label{eq:change-of-variable-T-structure}
    \psi_{\blambda}\colon R\dbb{t_I}\longrightarrow R\dbb{\bt_I} ,\quad t_i\mapsto \sum_{k\in K} \lambda_k t_{i,k} .
\end{align}
This induces a functor $(\cH ,\nabla) /R\dbb{t_I}\mapsto (\tcH ,\tnabla )/ \bbk\dbb{\bt_I}$ from $R$-linear (T)-structures to $\bbk$-linear (T)-structures.

\begin{lemma} \label{lifting lemma}
Let $R$ be a $\bbk$-algebra with a fixed $\bbk$-basis $\blambda = (\lambda_k , k\in K)$.
\begin{enumerate}[wide]
    \item There exists a functor $(\cH,\nabla)/ R\dbb{t_I}\mapsto (\tcH ,\tnabla )/ \bbk\dbb{\bt_I}$ from $R$-linear (T)-structures to $\bbk$-linear (T)-structures.
    It is obtained by applying the change of variable \eqref{eq:change-of-variable-T-structure} and forgetting the $R$-linear structure.
    \item Any framing for $(\cH ,\nabla ) /R\dbb{t_I}$ induces a framing for $(\tcH ,\tnabla ) / \bbk\dbb{\bt_I}$.
\end{enumerate}
    
\end{lemma}

\begin{proof}
    Let $(\cH ,\nabla )$ be an $R$-linear (T)-structure over $R\dbb{t_I}$.
    We define $\tcH$ to be the $\bbk\dbb{\bt_I}$-module obtained by forgetting the $R$-linear structure on $\cH\otimes_{R\dbb{t_I}} R\dbb{\bt_I}$.
    
    To define the (T)-structure connection $\tnabla$ we fix a trivialization $\cH\simeq H\otimes_R R\dbb{t_I}$.
    This induces an isomorphism $\tcH\simeq \tH \otimes_{\bbk}\bbk\dbb{\bt_I}$, where $\tH$ denotes the $\bbk$-module obtained from $H$ by forgetting the $R$-linear structure.
    We have a map of $\bbk\dbb{t_I,u}$-algebras
    \begin{align}
        \Psi_{\blambda}\colon \End_R (H)\dbb{t_I,u}\longrightarrow \End_{\bbk} (\tH)\dbb{\bt_I,u},
    \end{align}
    given by applying the change of variable $\psi_{\blambda}$ and forgetting the $R$-linear structure. 
    Fix $(i,k)\in I\times K$, and write $\nabla_{\partial_{t_i}} = \partial_{t_i} +  u^{-1}\bA_i (t_I,u)$, with $\bA_i (t_I,u)\in \End_{R} (H)\dbb{t_I,u}$.
    We then set 
    \[\tnabla_{\partial_{t_{i,k}}}\coloneqq \partial_{t_{i,k}} + u^{-1}\lambda_k \tbA_i (\bt_I ,u) ,\]
    where $\tbA_i \coloneqq \Psi_{\blambda} (\bA_i )$.
    The chain rule and the flatness of $\nabla$ imply that $\tnabla$ is flat, producing a $\bbk$-linear (T)-structure $(\tcH ,\tnabla)$ over $\bbk\dbb{\bt_I}$.
        It is easily checked that this (T)-structure is independent of the choice of trivialization for $(\cH ,\nabla )$.
    
    We now check functoriality.
    Let $(f,\phi)\colon (\cH
    ,\nabla ) /R\dbb{t_I}\rightarrow (\cH',\nabla' ) /R\dbb{s_J}$ be a morphism of (T)-structures.
    Let $(\tcH , \tnabla )/ \bbk\dbb{\bt_I}$ and $(\tcH' ,\tnabla') / \bbk\dbb{\bs_J}$ denote the induced $\bbk$-linear (T)-structures.
    There exists a unique morphism of $\bbk$-algebras $\tf\colon \bbk\dbb{\bs_J}\rightarrow \bbk\dbb{\bt_I}$ making the following diagram of $R$-algebras commutative 
    \[ 
    \begin{tikzcd}
        R\dbb{s_J} \ar[r,"f"] \ar[d,"\psi_{\blambda}'"] & R\dbb{t_I}\ar[d,"\psi_{\blambda}"] \\
        R\dbb{\bs_J} \ar[r, "\tf\otimes_{\bbk} 1"] & R\dbb{\bt_I} .
    \end{tikzcd}  
    \]
    It is characterized by the relations $\psi_{\blambda}\circ f(s_j) = \sum_{k\in K}\lambda_k \tf (s_{j,k} )$ for all $j\in J$, and is automatically continuous.
    The morphism of $R\dbb{t_I,u}$-modules $\phi\colon \cH\rightarrow \cH'\otimes_{R\dbb{s_J,u}} R\dbb{t_I,u}$ induces a morphism of $\bbk\dbb{\bt_I,u}$-modules $\tphi\colon\tcH \rightarrow \tcH'\otimes_{\bbk\dbb{\bs_J,u}} \bbk\dbb{\bt_I,u}$ obtained by forgetting the $R$-linear structure of the map of $R\dbb{\bt_I,u}$-modules
    \begin{align*}
       \cH\otimes_{R\dbb{t_I,u}}R\dbb{\bt_I,u}&\xrightarrow{\phi\otimes 1} \cH'\otimes_{R\dbb{s_J,u}}R\dbb{t_I,u}\otimes_{R\dbb{t_I,u}}R\dbb{\bt_I,u} \\
        &\simeq (\cH' \otimes_{R\dbb{s_J,u}} R\dbb{\bs_J,u} )\otimes_{R\dbb{\bs_J,u}} R\dbb{\bt_I,u} .
    \end{align*}
    Forgetting the $R$-linear structure, the right-hand side is naturally isomorphic to $\tcH'\otimes_{\bbk\dbb{\bs_J,u}}\bbk\dbb{\bt_I,u}$.
    Fixing trivializations of the (T)-structures, we directly check that the pair $(\tf,\tphi )$ is compatible with the connections.
    We omit the check that this is compatible with composition of morphisms.
    By construction, a framing trivialization for $(\cH ,\nabla )$ induces a framing trivialization for $(\tcH ,\tnabla)$, concluding the proof.    
    \end{proof}
\begin{remark}
The functor $(\cH,\nabla)/ R\dbb{t_I}\mapsto (\tcH,\tnabla)/ \bbk\dbb{\bt_I}$ defined above for (T)-structures is analogous to the composition of inverse image functor $\psi_{\blambda}^*$ and the restriction of scalars from $R$ to $\bbk$ in the theory of $D$-modules.
\end{remark}

\begin{definition}\label{def:R-linear-lift}
    An \emph{$R$-linear lift} of a $\bbk$-linear (T)-structure $(\cH,\nabla )/\bbk\dbb{\bt_I}$ is the data of an $R$-linear (T)-structure $(\cH_R ,\nabla_R ) / R\dbb{t_I}$ and an isomorphism of $\bbk$-linear (T)-structures $\alpha \colon (\cH ,\nabla)_0\xrightarrow{\sim} (\tcH_R , \tnabla_R)$.
\end{definition}

\subsection{Equivariant F-bundles} \label{subsec:eq-F-bundle}

\begin{definition}[Equivariant F-bundle] \label{def:eq-F-bundle}
    Let $I$ and $J$ be countable sets.
    An $R$-\emph{equivariant F-bundle} over $\bbk\dbb{\bt_I}$ consists of the following data $\lbrace (\cH ,\nabla ) , (\cH_R,\nabla_R) , \alpha\rbrace$.
    \begin{enumerate}
        \item $(\cH ,\nabla )$ is a $\bbk$-linear F-bundle over $\bbk\dbb{\bt_I}$, and
        \item $\alpha\colon (\cH,\nabla)_0\xrightarrow{\sim} (\tcH_R,\tnabla_R)$ is an $R$-linear lift of the underlying (T)-structure $(\cH ,\nabla)_0$, where $\cH_R$ has finite rank as a $R\dbb{t_I,u}$-module.
    \end{enumerate} 
    
    A \emph{morphism} of equivariant F-bundles 
    \[\lbrace (\cH,\nabla ) ,(\cH_R,\nabla_R ) ,\alpha\rbrace/ \bbk\dbb{\bt_I} \xrightarrow{(\beta,\beta_R)} \lbrace (\cH',\nabla') , (\cH_R',\nabla_R') ,\alpha' \rbrace / \bbk\dbb{\bt_J}\]
    consists of
    \begin{enumerate}
        \item a morphism of $\bbk$-linear F-bundles $\beta\colon (\cH,\nabla ) \rightarrow (\cH',\nabla')$, and
        \item a morphism of $R$-linear (T)-structures $\beta_R\colon (\cH_R,\nabla_R)\rightarrow (\cH_R',\nabla_R')$,
    \end{enumerate} 
    such that the following diagram of $\bbk$-linear (T)-structures commutes
    \begin{equation}\label{cd:morphism-eq-F-bundle}
    \begin{tikzcd}
        (\cH,\nabla )_0 \ar[r, "{\beta_0}"] \ar[d, "\alpha"] & (\cH',\nabla')_0 \ar[d,"\alpha'"] \\
        {(\tcH_R ,\tnabla_R)} \ar[r,"{\tbeta_R}"] & {(\tcH_R',\tnabla_R') .}
    \end{tikzcd}
    \end{equation}
    
\end{definition}

\begin{remark}\label{remark:when-R-is-k}
    \begin{enumerate}[wide]
        \item We identify a $\bbk$-linear F-bundle $(\cH ,\nabla )$ with the $\bbk$-equivariant F-bundle $\lbrace (\cH ,\nabla ),\allowbreak (\cH,\nabla)_0, \allowbreak \id\rbrace$, where we choose $1\in \bbk$ as a $\bbk$-basis of $\bbk$.
    This defines a fully faithful functor. 
        \item  When $\dim_{\bbk} R =1$, equivariant F-bundles correspond to $\bbk$-linear F-bundles of finite dimension and parametrized by finitely many variables, up to isomorphism.
    Indeed, after choosing the basis given by $1\in R$ the change of coordinate \eqref{eq:change-of-variable-T-structure} is the identity and the formal variables $t_I$ and $\bt_I$ agree.
    Given an equivariant F-bundle $\cF = \lbrace (\cH , \nabla ) , (\cH_{R},\nabla_{R}) ,\alpha\rbrace / \bbk\dbb{t_I}$, using $\alpha$ we see that $\cH$ has finite rank over $\bbk\dbb{t_I}$ because $\cH_{R}$ does, and we can define a $u$-direction connection on $\cH_{R}$ compatible with the (T)-structure, making $\alpha$ an isomorphism of F-bundles.
    \end{enumerate}
\end{remark}

\begin{definition} \label{def:framing-eq-F-bundle}
\begin{enumerate}[wide]
    \item A \emph{framing} for an equivariant F-bundle $\lbrace (\cH ,\nabla ) , (\cH_R,\nabla_R), \alpha\rbrace$ is the data of framings for $(\cH ,\nabla )$ and $(\cH_R,\nabla_R)$, such that $\alpha\colon (\cH,\nabla)_0\rightarrow (\tcH_R,\tnabla_R)$ is compatible with the induced framings.
    \item A morphism $(\beta,\beta_R)$ of framed equivariant F-bundles is \emph{compatible with the framings} if both $\beta$ and $\beta_R$ are compatible with the framings.
\end{enumerate}
\end{definition}

\begin{remark} \label{remark:morphism-eq-F-bundle}
    A morphism of equivariant F-bundles $(\beta ,\beta_R)$ is uniquely determined by $\beta_R$ and the $R$-linear lifts through \eqref{cd:morphism-eq-F-bundle}.
    Similarly, a framing of equivariant F-bundle is uniquely determined by the framing on the $R$-linear lift. 
\end{remark}

\section{Unfolding of equivariant F-bundles}

Recall the setting of \cref{sec:eq-F-bundles-and-T-structure}, $\bbk$ is a field of characteristic $0$ and $R$ is a $\bbk$-algebra of countable dimension.

\subsection{Framing of (T)-structures}

In this subsection, we prove that an F-bundle $(\cH,\nabla )$ is characterized by the underlying (T)-structure and the restriction of the F-bundle to a point using framing of (T)-structures (see \cref{lemma:lift-T-structure-morphism-to-F-bundle}). 
We deduce a criterion for lifting a morphism of (T)-structures to a morphism of F-bundles.
We also prove the existence of framing and extension of framing results for (T)-structures over a noetherian base.

\begin{lemma}\label{lemma:lift-T-structure-morphism-to-F-bundle} 
    For $k=1,2$, let $I_k$ be a countable set and $(\cH_k ,\nabla_k) /R\dbb{t_j ,j\in I_k}$ be an F-bundle.
    Let $(f,\Phi)\colon (\cH_1,\nabla_1)_0\rightarrow (\cH_2,\nabla_2)_0$ be a morphism of (T)-structures.
    Assume that the (T)-structure $(\cH_1,\nabla_1)_0$ admits a framing.     \begin{enumerate}[wide]
        \item $\nabla_1$ is uniquely determined by the underlying (T)-structure and $\nabla_{1,\partial_u}\vert_{t_{I_1} =0}$, and any such data determine a unique F-bundle connection extending the (T)-structure.
        \item $(f,\Phi)$ is an isomorphism of F-bundles if and only if $(f,\Phi)\vert_{t_{I_1}=0}$ is an isomorphism of F-bundles. 
    \end{enumerate}
\end{lemma}

\begin{proof}
    For (1), fix a framing trivialization $\cH\simeq H\otimes R\dbb{t_i ,i\in I_1,u}$ of the underlying (T)-structure.
    In this trivialization, write $\nabla_{1,\partial_{t_i}} = \partial_{t_i} + u^{-1} T^i$ and $\nabla_{1,\partial_u} = \partial_u + u^{-2} U$.
    By assumption, the endomorphism $T^i$ is independent of $u$.
    The flatness equations for the $u$-direction and $t_i$-direction give for all $i\in I_1$
    \begin{align}
    \label{eq: extending U direction}
    \frac{\partial U}{\partial t_i} = -T^i + u \frac{\partial T^i}{\partial u} + u^{-1} [U,T^i] = -T^i + u^{-1} [U,T^i] .
    \end{align}
    Any $U$ solving this system of equations gives rise to an F-bundle structure extending the (T)-structure.
    Then (1) reduces to proving that for any initial condition $U_0(u)\in \End_{R\dbb{u}}(H\dbb{u})$, there exists a unique $U(t,u)$ solving \eqref{eq: extending U direction} with $U(0,u) = U_0(u)$.
    Introduce the differential operators $D_i\colon X\mapsto \frac{\partial X}{\partial t_i} + u^{-1} \ad_{T^i} (X)$, where $\ad_{T^i} = [T^i , \cdot ]$.
    Then \eqref{eq: extending U direction} can be written as $D_i (U) = -T^i$, and we need to prove that the system is compatible for any initial condition.
    
    Since $\nabla_1$ is flat, by comparing degrees in $u$, we have for all $i,j\in I_1$
    \begin{align}\label{eq:flatness-framed-(T)-structure}
        [T^j,T^i]= u \left(\frac{\partial T^j}{\partial t_i}-\frac{\partial T^i}{\partial t_j}\right)=0.
    \end{align}
    It follows that 
    \begin{align*}
        [D_i,D_j] &= [\partial_{t_i},  \partial_{t_j}] + u^{-1}\big([\frac{\partial}{\partial t_i} , \ad_{T^j}] + [\ad_{T^i}, \frac{\partial }{\partial t_j} ]\big) + u^{-2}[\ad_{T^i} , \ad_{T^j}] \\
        &= u^{-1} (\ad_{\partial_{t_i} T^j}-\ad_{\partial_{t_j} T^i}) + u^{-2} \ad_{[T^i,T^j]} = 0.
    \end{align*}
    Hence, by the usual theory of linear system of ODEs, the system is compatible if and only if for all $i,j\in I_1$, we have $D_i (T^j)  =D_j (T^i)$.
    This follows from the flatness equations \eqref{eq:flatness-framed-(T)-structure}.
    We can thus construct a unique solution inductively on the number of variables from any initial condition.
    If $I_1$ is finite, we obtain a solution in finitely many steps. 
    If $I_1\simeq \bbN$ is infinite, we construct a solution in the projective limit ${\varprojlim} \End_R (H)\dbb{t_1,\dots, t_n,u} = \End_R (H)\dbb{t_i,i\in I_1,u}\simeq \End_{R\dbb{t_i,i\in I_1,u}} (H\otimes R\dbb{t_i, i\in I_1,u} )$.
    (1) is proved.
    
    For (2), the first direction is obvious. 
    For the converse, if $\Phi\vert_{t_{I_1} = 0}$ is an isomorphism, then the $R\dbb{t_{I_1}}$-module map $\Phi$ is an isomorphism (see \cref{remark:free-module-power-series}).
    The connection $\nabla_2' \coloneqq \Phi^{-1}\circ f^{\ast}\nabla_2\circ\Phi$ defines an F-bundle structure on $\cH$.
    By assumption, the underlying (T)-structure agrees with $(\cH ,\nabla )_0$ and $\nabla_{1,\partial_u}\vert_{t_{I_1}=0} =\nabla_{2,\partial_u}'\vert_{t_{I_1}=0}$.
    It follows from the uniqueness in (1) that $\nabla_2' = \nabla_1$, hence $(f,\Phi)$ is a morphism of F-bundles.    
\end{proof}

For (T)-structures defined over a Noetherian base $R\dbb{t_1,\dots , t_n}$, results from \cite[\S4.1]{HYZZ_decomposition_and_framing_of_F_bundles} imply the existence of framing trivializations. 

\begin{proposition} \label{lemma:weak-framing-t-directions}
    Let $(\cH ,\nabla )/R\dbb{t_1,\dots, t_n}$ be an $R$-linear (T)-structure.
    Any trivialization of $\cH\vert_{t=0}$ extends uniquely to a framing of $(\cH,\nabla )$.
\end{proposition}

\begin{proof}
    Fix a trivialization $\cH\simeq H\otimes R\dbb{t_1,\dots, t_n,u}$ lifting the trivialization of $\cH /(t_1,\dots, t_n)\cH$.
    Write the connection as $\nabla_{\partial_{t_i}} = \partial_{t_i} + u^{-1} T^i (t,u)$.
    We want to show that there exists a unique gauge transformation $P(t,u)\in \GL (H\dbb{t_1,\dots , t_n,u} )$ with $P(0,u) = \id$ such that $uP^{-1}\frac{\partial P}{\partial t_i} + P^{-1} T^i P$ is independent of $u$ for all $1\leq i\leq n$.
    This amounts to solving the system of PDEs ($1\leq i\leq n$)
    \[\frac{\partial P}{\partial t_i} = u^{-1} (-T^i P + PP_0^{-1} T_{-1}^iP_0 ),\]
    where $P_0 = P(t,0)$ and $T_{-1}^i = T^i (t,0)$, with the initial condition $P(0,u) = \id$.
    Uniqueness is clear, as the system provides recursive relations for the coefficients of $P$, and existence follows from \cite[Lemmas 4.17, 4.18, 4.20]{HYZZ_decomposition_and_framing_of_F_bundles}.
    The arguments there still apply, because we assume that $R$ contains $\bbQ$.
\end{proof}

Fix $I$ a finite set, let $(\cH,\nabla )/ R\dbb{t_I}$ be a (T)-structure of finite rank $n\in\bbN$. 
Let $v_1\in \cH / (t_I,u)\cH$.
Any choice $(h_1,\dots, h_n)$ of $R\dbb{t_I,u}$-basis for $\cH$ provides a trivialization through the isomorphisms 
\[\cH\simeq \bigoplus_{1\leq i\leq n} R\dbb{t_I,u}h_i \simeq R^{\oplus n}\otimes_R R\dbb{t_I,u} .\]
We call a basis $(h_1,\dots, h_n)$ \emph{good for $(\cH ,\nabla)$} if it induces a framing trivialization.
We say that it \emph{extends $v_1$} if $h_1$ is a lift of $v_1$.
\cref{lemma:weak-framing-t-directions} implies that any basis of $\cH / (t_I,u)\cH$ lifts uniquely to a good basis of $(\cH,\nabla)$.
More generally, we have the following.

\begin{lemma} \label{corollary: bundle map is uniquely determined on a point}
    Let $I$ and $J$ be finite sets.
    Let $(f,\Phi)\colon (\cH , \nabla ) / R\dbb{t_I}\rightarrow (\cH',\nabla') /R\dbb{t_J}$ be a morphism of finite rank (T)-structures. Assume that $\Phi\vert_{t_I = 0}$ is an isomorphism.
    \begin{enumerate}[wide]
        \item Any good basis $(h_1,\dots , h_n)$ of $(\cH,\nabla)$ induces a unique good basis $(h_1',\dots, h_n')$ of $(\cH',\nabla')$ such that $\Phi (h_k) = f^{\ast} (h_k')$ for all $1\leq k\leq n$.
        \item $\Phi$ is uniquely determined by its restriction to $\cH\vert_{t_I=0}$.
    \end{enumerate}
\end{lemma}

\begin{proof}
    The assumptions imply that $\Phi$ is an isomorphism of $R\dbb{t_I,u}$-modules.
    In particular, we have isomorphisms of $R\dbb{u}$-modules 
    \begin{align}\label{eq:iso-at-t=0}
        \cH/ (t_I)\cH\simeq f^{\ast}\cH' / (t_I)f^{\ast}\cH'\simeq \cH'/(t_J)\cH' .
    \end{align}
    A good basis $(h_1',\dots, h_n')$ for $(\cH',\nabla')$ is uniquely characterized by its projection to $\cH'/(t_J)\cH'$.
    This value is uniquely specified by the condition $\Phi (h_k ) = f^{\ast} (h_k')$ using the isomorphism \eqref{eq:iso-at-t=0}, which proves (1).

    For (2), we note that $\Phi$ is uniquely determined by the image of a good basis $(h_1,\dots , h_n)$ of $(\cH,\nabla )$.
    By (1), the image $(\Phi (h_1),\dots, \Phi (h_n))$ is a good basis for $f^{\ast} (\cH',\nabla')$.
    In particular, it is uniquely determined by its restriction to $t_I=0$, which only depends on $\Phi\vert_{t_I=0}$.
    The proof is complete.  
    \end{proof}

\subsection{Formal Hertling-Manin unfolding theorem}

In this subsection, we prove an analogue of the Hertling-Manin unfolding theorem for (TE)-structures (see \cite[Theorem 2.5]{Hertling_Unfoldings_of_meromorphic_connections}) for formal $R$-linear F-bundles and (T)-structures.

\begin{definition}[Unfolding of (T)-structure, F-bundle]
    Let $R$ be a $\bbk$-algebra, $I$ and $J$ countable sets. 
    Let $(\cH ,\nabla ) / R\dbb{t_I}$ be an $R$-linear (T)-structure (resp.\ F-bundle).
    An \emph{unfolding} of $(\cH ,\nabla )$ is a morphism of (T)-structures (resp.\ F-bundles) $(i,\phi)\colon (\cH ,\nabla ) /R\dbb{t_I}\rightarrow (\cH',\nabla' ) /R\dbb{t_J}$, where
    \begin{enumerate}
        \item $I\subset J$ and $i\colon R\dbb{t_J}\rightarrow R\dbb{t_I}$ is the quotient by the closure of the ideal $(t_j ,j \in J\setminus I)$, and
        \item $\phi\colon \cH\rightarrow i^{\ast}\cH'$ is an isomorphism of $R\dbb{t_I,u}$-modules.
    \end{enumerate}
    
    A \emph{morphism between two unfoldings} $\iota_k\colon (\cH ,\nabla)\rightarrow (\cH_k ,\nabla_k)$ for $k=1,2$, is a morphism of (T)-structures (resp. F-bundles) $(f,\psi)\colon (\cH_2,\nabla_2)\rightarrow (\cH_1,\nabla_1)$ such that $\psi$ is an isomorphism and the following diagram commutes
    \begin{equation*} \label{cd:induced-unfolding-F-bundle}
    \begin{tikzcd}
    & (\cH ,\nabla ) \ar[dl, "{\iota_2}" above left] \ar[dr, "{\iota_1}"] & \\
   (\cH_2,\nabla_2 ) \ar[rr,"{(f,\psi)}"] & & (\cH_1,\nabla_1 ).
    \end{tikzcd}
    \end{equation*}
    \end{definition}

\begin{remark}\label{remark:uniqueness-bundle-map-unfolding}
In the above commutative diagram, assume that $(\cH_k ,\nabla_k)$ depends on finitely many variables indexed by a finite set $J_k$ for $k=1,2$, and write 
\[\iota_k  =(i_k,\phi_k): (\cH,\nabla)/R\dbb{t_I}\rightarrow (\cH_2,\nabla_2)/R\dbb{t_{J_k}}.\] 
Then for any two morphisms $(f,\psi_k)$, $k=1,2$, between the unfoldings $\iota_2$ and $\iota_1$, we have $\psi_1=\psi_2$. 
In other words, the morphism on the base $f$ determines the bundle map.
Indeed, the commutativity of the diagram implies that $i_2^* \psi_k \circ \phi_2 = \phi_1.$ This determines $\psi_k|_{t_{J_2}=0} = \phi_1\circ \phi_2^{-1}|_{t_{J_2}=0}$. By \cref{corollary: bundle map is uniquely determined on a point}, $\psi_k$ is uniquely determined by $\psi_k|_{t_{J_2} = 0}$, thus $\psi_1=\psi_2$.
\end{remark}

\begin{remark} \label{remark:unfolding-framing} 
When $I$ and $J$ are finite, given an unfolding of $R$-linear (T)-structures
\[(i,\phi)\colon (\cH ,\nabla ) /R\dbb{t_I}\longrightarrow (\cH',\nabla' ) /R\dbb{t_J},\] 
any framing for $(\cH,\nabla)/R\dbb{t_I}$ induces a unique framing for $(\cH',\nabla')/R\dbb{t_J}$, and vice versa. Indeed, $\phi$ takes the framing trivialization for $(\cH,\nabla)$ to a framing trivialization for $i^* (\cH',\nabla')$, which is uniquely determined by its restriction to the fiber $i^*\cH'|_{t_I=0} = \cH'|_{t_J=0}$. We can extend this to a framing trivialization for $(\cH',\nabla')$ by \cref{lemma:weak-framing-t-directions}.
\end{remark}

\begin{lemma}\label{lemma:extend u-direction from T-structure}
For $k=1,2$, let $I_k$ be countable sets, and let \[(f,\Phi)\colon (\cH_1,\nabla_1)/R\dbb{t_1}\rightarrow (\cH_2,\nabla_2)/R\dbb{t_2}\] be an unfolding of $R$-linear (T)-structures.
Assume the (T)-structure $(\cH_2,\nabla_2)$ admits a framing.
Given an F-bundle structure $(\cH_1,\nabla_1^F)$ on $(\cH_1,\nabla_1)$, there exists a unique F-bundle structure $(\cH_2,\nabla_2^F)$ on $(\cH_2,\nabla_2)$ such that $(f,\Phi)$ is an unfolding of F-bundles.
\end{lemma}

\begin{proof}
    Since $(f,\Phi)$ is an unfolding of (T)-structures, we have isomorphisms of $R\dbb{u}$-modules
    \begin{equation} \label{eq:iso-fibers-unfolding}
        \cH_1\vert_{t_1=0}\simeq f^{\ast}\cH_2\vert_{t_1=0}\simeq \cH_2\vert_{t_2=0} .
    \end{equation}
    Under this isomorphism, the restriction $\nabla_1^F\vert_{t_1=0}$ produces a F-bundle connection on $\cH_2\vert_{t_2=0}$.
    Since the latter admits a framing, applying \cref{lemma:lift-T-structure-morphism-to-F-bundle}(1) we obtain a unique F-bundle $(\cH_2,\nabla_2^F)$ extending the (T)-structure $(\cH_2,\nabla_2 )$.
    We now check that $(f,\Phi)$ is a morphism of F-bundles.
    By construction, the connections $f^{\ast}\nabla_2^F$ and $\Phi\circ\nabla_1^F\circ\Phi^{-1}$ are F-bundle connections on $f^{\ast}\nabla_2$ which coincide at $t_1=0$, and with the same underlying (T)-structures.
    The framing for $(\cH_2,\nabla_2)$ induces a framing on $f^{\ast} (\cH_2,\nabla_2)$, as can be seen by fixing a framing trivialization of $(\cH_2,\nabla_2)$ and pulling it back under $f$.
    Then, it follows from \cref{lemma:lift-T-structure-morphism-to-F-bundle}(1) that those two F-bundle structures agree. 
    Hence, $(f ,\Phi)$ is a morphism of F-bundles. 

    For uniqueness, note that the F-bundle connection $\nabla_2^F$ is uniquely determined by its restriction to $t_2=0$ since $(\cH_2,\nabla_2)$ admits a framing, and that $\nabla_2^F\vert_{t_2=0}$ is uniquely specified by $\nabla_1^F\vert_{t_1=0}$ through the isomorphisms \eqref{eq:iso-fibers-unfolding}.
\end{proof}

For an $R$-linear (T)-structure $(\cH ,\nabla ) / R\dbb{t_I}$, there is a morphism of $R$-modules \cite[Remark 2.3]{HYZZ_decomposition_and_framing_of_F_bundles}
\begin{align}\label{eq:map-to-endomorphism-fiber}
    \mu\colon \bigoplus_{i\in I} R\partial_{t_i} &\longrightarrow \End_R (H), \\
    \partial_{t_i} &\longmapsto \nabla_{u\partial_{t_i}}\big\vert_{u=0, t_I=0}, \nonumber
\end{align}
where $H \coloneqq \cH / J\cH$ with $J$ the closure of the ideal $(t_I,u)$.
For each $v\in H$ we obtain an evaluation map of $R$-modules:
\begin{align} \label{eq:evaluation-map}
    \mu_v\colon \bigoplus_{i\in I} R\partial_{t_i} &\longrightarrow H, \\
     \xi &\longmapsto \mu (\xi) (v) . \nonumber
\end{align}
Furthermore, if $(\cH,\nabla )$ is an F-bundle, we also have a residue endomorphism in the $u$-direction $\mbox{K}\coloneqq [u^2\nabla_{\partial_u}]\vert_{u=t=0}\in\End_R (H)$.
We introduce the notion of maximal (T)-structure and maximal F-bundle, analogous to \cite[Definition 2.6]{HYZZ_decomposition_and_framing_of_F_bundles}.

\begin{definition}[Maximal (T)-structure, maximal F-bundle] \label{def:maximal-F-bundle}
    Let $R$ be a $\bbk$-algebra, $I$ a countable set, and $J\subset R\dbb{t_I,u}$ the closure of the ideal $(t_I,u)$.
    An $R$-linear (T)-structure, or F-bundle, $(\cH,\nabla ) /R\dbb{t_I}$ is \emph{maximal} if there exists $v\in \cH /J\cH$ such that the map $\mu_v$ is an isomorphism.
    We call such a $v$ a \emph{cyclic vector}.
\end{definition}

The Hertling-Manin unfolding theorem guarantees the existence and uniqueness of a maximal unfolding under certain conditions, which we introduce in the next definition.

\begin{definition} \label{def:IC-GC-conditions}
    Let $I$ be a countable set, $(\cH,\nabla ) / R\dbb{t_I}$ an $R$-linear (T)-structure (resp.\ F-bundle), and $J\subset R\dbb{t_I,u}$ the closure of the ideal $(t_I,u)$.
    We define the following conditions on an element $v\in H\coloneqq \cH / J\cH$:
    \begin{itemize}
        \item[(IC)] The map $\mu_v$ in \eqref{eq:evaluation-map} is injective. 
        \item[(GC)] The orbit of $v$ under the action of the subalgebra $R[\im \mu]\subset \End_R (H)$ (resp.\ $R[\im\mu , \mbox{K}]\subset \End_R (H)$) defined by evaluation on $v$ is $H$.
        \item[(GC')] The condition (GC) is satisfied after base change to $\Frac (R)$. 
    \end{itemize}
    If $v$ satisfies (GC), we say that $v$ is a \emph{generating vector} for $(\cH,\nabla)$.
\end{definition}

The following lemma provides a construction of unfoldings under the (GC) condition. 
It is analogous to \cite[Lemma 2.9]{Hertling_Unfoldings_of_meromorphic_connections}, except that we use framings of (T)-structures to avoid the analytic argument used there.

\begin{lemma} \label{lemma:existence-uniqueness-unfoldings}
    Let $(\cH^{(0)} ,\nabla^{(0)} )/R\dbb{t_1,\dots, t_d}$ be an F-bundle of rank $n$ satisfying the (GC) condition, 
    let $v_1\in\cH^{(0)}/(t_I,u)\cH^{(0)}$ be a generating vector.
    Let $(h_1^{(0)},\dots, h_n^{(0)})$ be a good basis of $(\cH^{(0)} ,\nabla^{(0)})$ extending $v_1$.
    Fix $\ell\geq 1$ and let $f_1,\dots, f_n\in R\dbb{t_1,\dots, t_d,s_1,\dots , s_{\ell}}$ whose restrictions to $s=0$ are $0$.

    Then there exists an unfolding $\iota\colon(\cH^{(0)} ,\nabla^{(0)} ) /R\dbb{t_1,\dots, t_d}\rightarrow (\cH ,\nabla ) / R\dbb{t_1,\dots, t_d,s_1,\dots, s_{\ell}}$ such that, if $(h_1,\dots, h_n)$ denotes the good basis of $(\cH,\nabla)$ induced from $(h_1^{(0)} ,\dots , h_n^{(0)})$ (see Lemma \ref{corollary: bundle map is uniquely determined on a point}), we have for $1\leq j\leq \ell$
    \begin{equation}\label{eq:construction-unfolding}
        [u\nabla_{\partial_{s_j}}]\vert_{u=0} (h_1\vert_{u=0}) = \sum_{i=1}^n \frac{\partial f_i}{\partial s_j} h_i\vert_{u=0} . 
    \end{equation}
    Any two unfoldings satisfying \eqref{eq:construction-unfolding} are isomorphic under a morphism $(\id ,\psi)$, where $\psi$ identifies the canonical extensions of the good basis $(h_i^{(0)})_{1\leq i\leq n}$.
\end{lemma}

\begin{proof}
    Set $t\coloneqq \lbrace t_1,\dots ,t_d\rbrace$.
    We consider the case $\ell=1$, as we can always decompose an unfolding as a sequence of $1$-dimensional unfoldings.

    Let $H\coloneqq R^{\oplus n}$. The good basis $(h_i^{(0)})_{1\leq i\leq n}$ provides an isomorphism $\phi\colon \cH^{(0)}\xrightarrow{\sim} H\otimes_R R\dbb{t,u}$.
    Let $\cH\coloneqq H\otimes_{R} R\dbb{t,s,u}$.
    We first prove that there exists a unique connection $\nabla$ on $\cH$ such that $\iota = (i , \phi)\colon (\cH^{(0)} , \nabla^{(0)} )\rightarrow (\cH ,\nabla )$ is an unfolding satisfying \eqref{eq:construction-unfolding}.
    This is equivalent to constructing unique matrices $T^i (t,s),S(t,s),U_k (t,s)\in \Mat (n\times n, R\dbb{t,s} )$ such that the connection form 
    \[\Omega \coloneqq \frac{1}{u}\sum_{i=1}^n T^i(t,s)d t_i + \frac{1}{u} S(t,s) ds +\frac{1}{u^2} \sum_{k\geq 0} U_{k-2} (t,s)u^k du \]
    satisfies: 
    \begin{enumerate}[wide]
        \item[(a)] the flatness equation $d\Omega + \Omega\wedge\Omega = 0$,
        \item[(b)] $T^i (t,0)$ and $U_k (t,0)$ coincide with the connection matrix of $\nabla^{(0)}$ in $(h_i^{(0)} )_{1\leq i\leq n}$, and
        \item[(c)] $S(t,s)e_1 = \sum_{i=1}^n \frac{\partial f_i}{\partial s} e_i$, where $(e_i)_{1\leq i\leq n}$ is the canonical basis of $\cH/(u)\cH = R^{\oplus n}\otimes_R R\dbb{t,s}$.
    \end{enumerate}
        We further decompose the matrices into powers of $s$, and write $T_w^i (t)$ (resp.\ $S_w (t)$, $U_{k,w} (t)$) for the coefficient of $s^w$ in $T^i (t,s)$ (resp.\ $S (t,s)$, $U_k(t,s)$).
    We will construct the matrices order by order in $s$.

    Condition (a) is equivalent to the following system of equations:
    \begin{align}
        [S,T^i] &= 0 \label{eq:comm-S-T}\\
        [S, U_{-2}] &= 0 \label{eq:comm-S-U}\\
        \partial_s T^i &= \partial_{t_i} S \label{eq:derivative-S-T}\\
        \partial_{s} U_{-2} &= [U_{-1} , S] - S \label{eq:derivative-S-U-residue} \\
        \partial_{s} U_{k} &= [U_{k+1} , S] \qquad (k\geq -1) \label{eq:derivative-S-U-general} \\
        [T^i , T^j] &= 0 \label{eq:comm-T-T}\\
        [U_{-2} , T^i] &= 0 \label{eq:comm-T-U}\\
        \partial_{t_i} T^j &= \partial_{t_j} T^j \label{eq:derivative-T-T} \\
        \partial_{t_i} U_{-2} &= [U_{-1} , T^i] - T^i \label{eq:derivative-T-U-residue} \\
        \partial_{t_i} U_{k} &= [U_{k+1} , T^i] \qquad (k\geq -1) \label{eq:derivative-T-U-general}
    \end{align}
    
    We prove by induction on $m\in\bbN$ that there exists unique matrices $T_w^i (t)$ and $U_{k,w} (t)$ for $0\leq w\leq m$ and $S_w (t)$ for $0\leq w\leq m-1$ 
    such that the equations (\ref{eq:comm-S-T}) through (\ref{eq:derivative-S-U-general}) are satisfied modulo $s^m$, the equations (\ref{eq:comm-T-T}) through (\ref{eq:derivative-T-U-general}) are satisfied modulo $s^{m+1}$, condition (b) is satisfied and condition (c) is satisfied modulo $s^m$.

    For $m=0$, condition (b) provides the matrices $T_0^i (t)$ and $U_{k,0} (t)$, and the equations (\ref{eq:comm-T-T})-(\ref{eq:derivative-T-U-general}) are satisfied modulo $s$ by flatness of $\nabla^{(0)}$.

    Now assume the induction carried out until step $m$, we prove step $m+1$.
    We only need to construct the matrices $T_{m+1}^i$, $U_{k,m+1}$ and $S_m$ so that the various conditions of the induction are satisfied.
    The construction of a unique matrix $S_m$ such that (\ref{eq:comm-S-T}), (\ref{eq:comm-S-U}) and condition (c) are satisfied modulo $s^{m+1}$ is as in (i) of the proof of \cite[Lemma 2.9]{Hertling_Unfoldings_of_meromorphic_connections}.
    The matrices $T_{m+1}^i$ and $U_{k,m+1}$ are uniquely determined by imposing equations (\ref{eq:derivative-S-T})-(\ref{eq:derivative-S-U-general}) modulo $s^{m+1}$. \\
    It remains to check that equations (\ref{eq:comm-T-T})-(\ref{eq:derivative-T-U-general}) hold modulo $s^{m+2}$, assuming that equations (\ref{eq:comm-S-T})-(\ref{eq:derivative-T-U-general}) hold modulo $s^{m+1}$.
    Since they hold at $s=0$, we simply check that the $s$-derivative of these equations is zero modulo $s^{m+1}$.
    For (\ref{eq:comm-T-T}) we have modulo $s^{m+1}$
    \begin{align*}
        \partial_s [T^i, T^j] &= [\partial_s T^i , T^j] + [T^i ,\partial_s T^j] \\
        &= [\partial_{t_i} S, T^j] + [T^i ,\partial_{t_j} S] \\
        &= - [S, \partial_{t_i} T^j] - [\partial_{t_j} T^i ,S] \\
        &= 0.
    \end{align*}
    For (\ref{eq:comm-T-U}) we have modulo $s^{m+1}$
    \begin{align*}
        \partial_s [U_{-2}, T^i] &= [\partial_s U^{-2} , T^i] + [U_{-2} , \partial_s T^i] \\
        &= [[U_{-1} , S], T^i] + [U_{-2}, \partial_{t_i} S] \\
        &= [[U_{-1} ,S] , T^i] - [\partial_{t_i} U_{-2} , S]\\
        &= [[U_{-1}, S] ,T^i] - [[U_{-1}, T^i], S] \\
        &= 0 .
    \end{align*}
    For \eqref{eq:derivative-T-T} we have modulo $s^{m+1}$
    \[\partial_s(\partial_{t_i}T^j -  \partial_{t_j}T^i) = \partial_{t_i}\partial_s T^j -\partial_{t_j}\partial_s T^i  = \partial_{t_i}\partial_{t_j}S  - \partial_{t_j}\partial_{t_i}S =0 .\]
    For \eqref{eq:derivative-T-U-residue} we have modulo $s^{m+1}$
    \begin{align*}
        \partial_s \left( \partial_{t_i} U_{-2} + T^i + [T^i ,U_{-1}]\right) &= \partial_{t_i} [U_{-1} , S] - \partial_{t_i} S + \partial_{t_i} S +  [\partial_sT^i ,U_{-1}] + [T^i \partial_s U_{-1} ]\\
        &= [\partial_{t_i} U_{-1} , S] + [U_{-1} , \partial_{t_i} S] + [\partial_{t_i} S , U_{-1} ] + [T^i , [U_0 , S]] \\
        &= [[U_0 , T^i] , S] + [T^i , [U_0 ,S]] \\
        &= 0,
    \end{align*}
    where on the first line we used \eqref{eq:derivative-S-T} and \eqref{eq:derivative-S-U-residue}, on the second line we used \eqref{eq:derivative-S-T} and \eqref{eq:derivative-S-U-general}, on the third line we used \eqref{eq:derivative-T-U-general}, and on the last line we used the Jacobi identity and \eqref{eq:comm-S-T}.
    For (\ref{eq:derivative-T-U-general}) we have modulo $s^{m+1}$
    \begin{align*}
        \partial_s \left(\partial_{t_i} U_k + [T^i , U_{k+1}] \right) &= \partial_{t_i} [U_{k+1} , S] + [\partial_s T^i , U_{k+1}] + [T^i ,\partial_s U_{k+1}] \\
        &= [\partial_{t_i} U_{k+1} ,S] + [T^i , \partial_s U_{k+1}] \\
        &= [[U_{k+2} ,T^i], S] + [T^i ,[U_{k+2}, S]] \\
        &= 0.
    \end{align*}
    This finishes the induction step, and proves the existence.

    For uniqueness up to isomorphism, assume that $\iota'\colon (\cH^{(0)} ,\nabla^{(0)} )\rightarrow (\cH ',\nabla ' )$ is another unfolding satisfying \eqref{eq:construction-unfolding}.
    We prove that it is isomorphic to the unfolding $(\cH,\nabla )$ constructed above.
    Let $\psi\colon \cH\rightarrow \cH'$ denote the $R\dbb{t,s,u}$-module isomorphism obtained by identifying the good bases obtained from $(h_i^{(0)} )_{1\leq i\leq n}$.
    Then the connection form of $\psi^{-1}\circ \nabla'\circ \psi$ in the trivialization of $\cH$ given by $(e_1,\dots, e_n)$ satisfies conditions (a), (b) and (c) above.
    Thus $\psi^{-1}\circ\nabla'\circ \psi = \nabla$, and we conclude that $(\id \psi)\colon (\cH ,\nabla )\rightarrow (\cH' ,\nabla' )$ is an isomorphism of unfoldings.
\end{proof}

\cref{lemma:existence-uniqueness-unfoldings} says that under the (GC) assumption, an unfolding $\iota\colon (\cH^{(0)},\nabla^{(0)})\rightarrow (\cH ,\nabla )$ is uniquely determined up to isomorphism by the choice of a good basis $(h_1,\dots, h_n)$ extending a cyclic vector, and the action of the connection on $h_1$.

\begin{theorem}[Hertling-Manin for F-bundles] \label{thm:Hertling-Manin-F-bundles}
    Let $R$ be an integral domain containing $\bbQ$.
    Let $(\cH ,\nabla ) /R\dbb{t_1,\dots, t_d}$ be a finite rank F-bundle. 
    Let $v\in \cH/(t_1,\dots, t_d,u)\cH$.
    \begin{enumerate}[wide]
        \item If $v$ satisfies (IC), (GC) and $\coker\mu_v$ is free, then there exists a maximal unfolding with cyclic vector induced from $v$.
        \item If $v$ satisfies (GC'), then any two maximal unfoldings of $(\cH,\nabla )$ with cyclic vector induced from $v$ are isomorphic under a unique isomorphism.
    \end{enumerate}
    Furthermore, any framing for $(\cH,\nabla )$ induces a unique framing on a maximal unfolding.
\end{theorem}

\begin{proof}
    Let $n$ denote the rank of $\cH$, and $\ell := n-d$.
    We assume $\ell\geq 0$, as otherwise the evaluation map $\mu_v$ cannot be injective and a maximal unfolding of $(\cH ,\nabla )$ does not exist.
    Write $t= \lbrace t_1,\dots, t_d\rbrace$ and $s = \lbrace s_1,\dots ,s_{\ell} \rbrace$.
    Fix a good basis $(h_1,\dots, h_n)$ for $(\cH ,\nabla )$ extending $v$, i.e. with $h_1\vert_{t=u=0} =v$.

    For (1), let $N\in \Mat (n\times d , R)$ denote the matrix of the evaluation map $\mu_v$.
    Let $f_1,\dots, f_n\in R\dbb{t,s}$ with $f_i(t,0) = 0$.
    Applying \cref{lemma:existence-uniqueness-unfoldings} we obtain an unfolding $\iota\colon (\cH,\nabla ) / R\dbb{t} \rightarrow (\cH',\nabla' ) /R\dbb{t,s}$.
    Let $v'\in \cH '/ (t,s,u)\cH'$ corresponding to $v$, the matrix of the evaluation map $\mu_{v'}$ in the good basis obtained from $(h_i)_{1\leq i\leq n}$ is 
    \begin{equation}\label{eq:matrix-maximal-unfolding}
        \begin{pmatrix}
        N &  (\frac{\partial f_i}{\partial s_j}\big\vert_{t=s=0})_{1 \leq i\leq n,1\leq j\leq \ell} 
    \end{pmatrix} \in \Mat (n\times n, R).
    \end{equation}
    Since $v$ satisfies (IC), the columns of $N$ form a basis of $\im\mu_v\subset\cH /(t_1,\dots, t_d,u)\cH$.
    Since $\coker\mu_v$ is free, by the basis extension theorem, we can extend this basis to a basis of $\cH /(t_1,\dots, t_d,u)\cH$ by adding elements $\{v_1,\dots,v_\ell \}$. 
    Any choice $(f_1,\dots, f_n)$ such that the vector $(\frac{\partial f_i}{\partial s_k}\vert_{t=s=0})_{1\leq i\leq n}$ corresponds to $v_k$ for all $1\leq k\leq \ell$ gives rise to a maximal unfolding, since the columns of \eqref{eq:matrix-maximal-unfolding} then form a basis of $H$.
    This proves (1).

    We now prove (2).
    For $k=1,2$, let $\iota_k = (i_k,\phi_k)\colon (\cH,\nabla ) \rightarrow (\cH_k',\nabla_k')$ be a maximal unfolding.
    In the good bases induced from $(h_i)_{1\leq i\leq n}$ the $1$-forms defining the (T)-structures are closed by \eqref{eq:derivative-T-T}, hence can be written as $u^{-1} dA_k$ for a unique $A_k\in \Mat (n\times n ,R)\dbb{t,s}$ satisfying $A_k(0,0) = 0$.
    The first column of $A_k$ provides $n$ elements of $R\dbb{t,s}$ that define a map of $R$-algebras $\psi_k\colon R\dbb{t,s}\rightarrow R\dbb{t,s}$.  
    Since the unfoldings are assumed to be maximal,  $d\psi_k\vert_{t=s=0}$ is an isomorphism. This follows from the fact that, by construction, its matrix in the basis $(dt_1,\dots, dt_d,ds_1,\dots , ds_{\ell} )$ coincides with the matrix of the evaluation map for $(\cH',\nabla' )$.
    We deduce that $\psi_k\in \Aut_R (R\dbb{t,s})$.
    If $(f,j)\colon (\cH_1,\nabla_1)\rightarrow (\cH_2,\nabla_2)$ is an isomorphism of unfoldings, then $f^{\ast}dA_2 = dA_1$ which implies $A_2\circ f = A_1$.
    In particular $\psi_2\circ f = \psi_1$, and this determines $f$ uniquely, since $\psi_2$ is an isomorphism. 
    In turn, this determines $j$ uniquely by \cref{remark:uniqueness-bundle-map-unfolding}.
    Conversely, let $f= \psi_2^{-1}\circ\psi_1$ and define $j\colon \cH_1'\rightarrow f^{\ast}\cH_2'$ by identifying the good bases induced from $(h_i)_{1\leq i\leq n}$. 
    In particular, we have $d\psi_1 = d\psi_2\circ df$, therefore $f^{\ast} (\cH_2',\nabla_2')$ is a maximal unfolding whose action on the cyclic section that extends $h_1$ agrees with that of $(\cH_1',\nabla_1')$.
    After base changing to $\Frac (R)$, the (GC) condition is satisfied.
    It follows from \cref{lemma:existence-uniqueness-unfoldings} that $(f,j)$ is compatible with the connections and is an isomorphism of unfoldings after base changing to $\Frac (R)$.
    But $f$ (resp. $j$) is invertible over $R$ (resp. $R\dbb{t,s,u}$) by construction, so the unfoldings are isomorphic over $R$.
      
    The last claim follows from the extension of framing result in \cite[Theorem 1.3]{HYZZ_decomposition_and_framing_of_F_bundles}.
    The proof is complete.
\end{proof}

\begin{corollary}[Hertling-Manin for (T)-structures]\label{thm:hertling-manin-T-structure}
    Let $R$ be an integral domain containing $\bbQ$.
    Let $(\cH ,\nabla ) /R\dbb{t_1,\dots, t_d}$ be a finite rank (T)-structure. 
    Let $v\in \cH/(t_1,\dots, t_d,u)\cH$.
    \begin{enumerate}[wide]
        \item If $v$ satisfies (IC), (GC) and $\coker\mu_v$ is free, then there exists a maximal unfolding with cyclic vector induced from $v$.
        \item If $v$ satisfies (GC'), then any two maximal unfoldings of $(\cH,\nabla )$ with cyclic vector induced from $v$ are isomorphic under a unique isomorphism.
    \end{enumerate}
\end{corollary}

\begin{proof}
    Let $n$ denote the rank of $\cH$. 
    Write $t = \lbrace t_1,\dots, t_d\rbrace$ and $s =\lbrace s_1,\dots ,s_{n-d}\rbrace$.
    We choose an F-bundle structure $(\cH ,\nabla^F)/R\dbb{t}$ lifting the (T)-structure $(\cH ,\nabla)$.
    Then $(\cH,\nabla^F)$ satisfies the conditions of \cref{thm:Hertling-Manin-F-bundles}(1), producing a maximal unfolding of F-bundle. 
    Since being maximal is a property of the (T)-structure, the unfolding of the underlying (T)-structure is maximal, proving (1).
        
    For (2), let $\iota_1\colon (\cH,\nabla ) /R\dbb{t}\rightarrow (\cH_1,\nabla_1)/R\dbb{t,s}$ and $\iota_2 \colon (\cH,\nabla)/R\dbb{t}\rightarrow(\cH_2,\nabla_2)/R\dbb{t,s}$ be two maximal unfoldings of (T)-structures, with cyclic vector induced from $v$. 
    Since the base of $(\cH ,\nabla )$ has finitely many variables, it follows from \cref{lemma:weak-framing-t-directions} that it admits a framing. 
    This induces a framing on any unfolding by \cref{remark:unfolding-framing}.
    Thus, we can apply \cref{lemma:extend u-direction from T-structure} and extend the two unfoldings $\iota_1$ and $\iota_2$ uniquely to maximal unfoldings of the F-bundle $(\cH,\nabla^F)$.
    We conclude from \cref{thm:Hertling-Manin-F-bundles} that they are isomorphic under a unique isomorphism, hence the same holds for the underlying unfoldings of (T)-structures.
    This concludes the proof.
\end{proof}

\begin{remark}[Existence when $R$ is not a field]\label{remark:existence-max-unfolding}
    Let $R$ be an integral domain,   $(\cH,\nabla )/ R\dbb{t_1,\dots, t_d}$ be a finite rank F-bundle, and  $v\in H\coloneqq \cH/(t_1,\dots, t_d,u)\cH$.
    \begin{enumerate}[wide]
        \item If $v$ only satisfies (IC) and (GC'), we know that a maximal unfolding exists after base change to $\Frac (R)$.
        In fact, the maximal unfolding is defined over any localization $R'$ of $R$ such that $\coker\mu_v\otimes_R R'$ is a free module, by \cref{thm:Hertling-Manin-F-bundles}(1).
        \item Let $(\cH,\nabla ) \rightarrow (\cH',\nabla' )$ be an unfolding. 
        We obtain maps $\mu$ and $\mu'$ as in \eqref{eq:map-to-endomorphism-fiber}.
        Let $\cA\coloneqq R[\im\mu]$ and $\cA'\coloneqq R[\im\mu']$ denote the associated commuting subalgebras of $\End_R (H)$.
        We have $\cA\subset \cA'\subset \cC (\cA')\subset \cC (\cA)$, where $\cC(\cdot )$ denotes the commutant algebra.
        Let $\tmu_v\colon \cA\rightarrow H$ and $\tmu_v'\colon \cA'\rightarrow H$ denote the evaluation on $v$.
        From the commutative diagram
        \[\begin{tikzcd}
            0 \ar[r] & \cA \ar[d,"\tmu_v"] \ar[r] & \cA' \ar[d,"\tmu_v'"]\ar[r]  & \cA' / \cA \ar[d]\ar[r] & 0 \\
            0 \ar[r] & H\ar[r, "\id"] & H \ar[r] & 0, &
        \end{tikzcd}\]
        we obtain the long exact sequence 
        \[0\longrightarrow \ker\tmu_v\longrightarrow \ker\tmu_v'\longrightarrow \cA '/ \cA \longrightarrow \coker\tmu_v \longrightarrow \coker \tmu_v'\longrightarrow 0 .\]
        If the unfolding is maximal, we have $\cA' = \im\mu'$ and $\tmu_v'$ is an isomorphism. 
        We deduce that $\ker\tmu_v = 0$ and $\coker\tmu_v\simeq \cA'/\cA$.
        Then, $v$ satisfies the (IC) condition but not necessarily the (GC) condition.
        In the special case when $\cA = \cC (\cA)$, a maximal unfolding exists if and only if $v$ satisfies (IC) and (GC).
    \end{enumerate}
    This is illustrated in \cref{example:full-not-maximal}.
\end{remark}

\begin{example}\label{example:full-not-maximal}
    Let $R = \bbk\dbb{\lambda_1,\lambda_2}$, $H = R^{\oplus 3}$ and $\cH = H\otimes_R R\dbb{t_1,t_2}$.
    Let $(e_1,e_2,e_3)$ denote the canonical basis of $H$. We consider the matrices
    \[A = \Id_3 , \; B = \begin{pmatrix} 0 & 0 & 1 \\ \lambda_1 & 0 & 0 \\ 0 & \lambda_2 & 0\end{pmatrix} , \; C = B^2 = \begin{pmatrix}
        0  & \lambda_2 & 0 \\ 0 & 0 & \lambda_1 \\ \lambda_1\lambda_2 & 0 & 0 
    \end{pmatrix} .\]
    Assume $\nabla$ is an F-bundle connection on $\cH$ such that $\mu(\partial_{t_1} ) = A$ and $\mu (\partial_{t_2}) = B$.
    We have $R[\im\mu] = RA\oplus RB\oplus RC$ and $R[\im\mu] = \cC (R[\im\mu])$.
    It follows from \cref{remark:existence-max-unfolding}(2) that there exists a maximal unfolding with cyclic vector $v = \alpha e_1+\beta e_2 + \gamma e_3$ if and only if $v$ satisfies (IC) and (GC).
    The matrix of the evaluation map $\tmu_v\colon R[\im \mu]\rightarrow H$ with respect to the bases $(A,B,C)$ and $(e_1,e_2,e_3)$ is
    \[\tmu_v = \begin{pmatrix}
        \alpha  & \gamma & \lambda_2 \beta \\
        \beta & \lambda_1\alpha & \lambda_1\gamma \\
        \gamma & \lambda_2 \beta & \lambda_1\lambda_2\alpha
    \end{pmatrix},\]
    whose determinant is $\lambda_1^2\lambda_2\alpha^3 + \lambda_2^2\beta^3 + \lambda_1\gamma^3 -3\lambda_1\lambda_2\alpha\beta\gamma$.
    The vector $v$ satisfies (IC) and (GC) if and only if this determinant is invertible.
    For $v=e_3$, this determinant is $\lambda_1$ and we conclude that the associated maximal unfolding is defined over $\bbk\dbb{\lambda_1,\lambda_2}[\lambda_1^{-1}]$.
    For $v=e_2$, this determinant is $\lambda_2^2$ and the associated maximal unfolding is defined over $\bbk\dbb{\lambda_1,\lambda_2}[(\lambda_2^2)^{-1}]$.
\end{example}

\subsection{Unfolding theorem for equivariant F-bundles}

In this subsection, we prove the unfolding theorem for equivariant F-bundles. 
The strategy is to unfold the $R$-linear (T)-structure using \cref{thm:hertling-manin-T-structure}, and then extend it in the $u$-direction using \cref{lemma:lift-T-structure-morphism-to-F-bundle}.

\begin{definition} \label{def:unfolding_of_equivariant_F-bundle}
    Let $R$ be a $\bbk$-algebra, and let $I$ be a countable set.
    \begin{enumerate}[wide]
        \item An \emph{unfolding of $\bbk$-linear equivariant F-bundle} $\lbrace (\cH,\nabla) ,(\cH_R,\nabla_R) ,\alpha\rbrace / \bbk\dbb{\bt_I}$ is a morphism of equivariant F-bundles $(\iota,\iota_R)$ such that $\iota$ is an unfolding of $\bbk$-linear F-bundles and $\iota_R$ is an unfolding of $R$-linear (T)-structure. 
        In particular, $\iota$ and $\iota_R$ are compatible with the $R$-linear lifts as in \eqref{cd:morphism-eq-F-bundle}.
        \item A \emph{morphism of unfoldings} is a morphism $(\beta ,\beta_R)$ of equivariant F-bundles such that both $\beta$ and $\beta_R$ are morphisms of unfoldings.
        In particular, $(\beta,\beta_R)$ commutes with the unfolding maps. 
        \item An equivariant F-bundle is \emph{maximal} if the underlying $R$-linear (T)-structure is maximal.
    \end{enumerate}
\end{definition}

\begin{lemma} \label{lemma:framing-T-structure-to-F-bundle}
    Let $I$ be a countable set.
    Let $(\cH ,\nabla ) / R\dbb{t_I}$ be an F-bundle. 
    A framing for the (T)-structure $(\cH ,\nabla )_0$ is a framing for the F-bundle if and only if it restricts to a framing of F-bundles at $t_I=0$.
\end{lemma}

\begin{proof}
    The framing provides a trivialization $\cH\simeq H\otimes_R R\dbb{t_I,u}$.
    Write $\nabla_{\partial_{t_i}} = \partial_{t_i} + u^{-1} T_i (t)$ and $\nabla_{\partial_u} = \partial_u + u^{-2} U(t,u)$.
    By \cref{lemma:lift-T-structure-morphism-to-F-bundle}(1), $U(t,u)$ is uniquely determined by the system of differential equations \eqref{eq: extending U direction} and the initial condition $U(0,u)$.
    Write $U (t,u) = \sum_{k\geq 0} U_{k-2} (t) u^k$.
    The differential equation implies for all $k\geq 0$
    \[\frac{\partial U_k}{\partial t_i} = -[T_i, U_{k+1}] .\]
    Since we have the initial condition $U_k (0) = 0$, we deduce that $U_k (t) = 0$ for all $k\geq 0$ by applying \cite[Lemma 4.8(1)]{HYZZ_decomposition_and_framing_of_F_bundles} inductively on the number of variables.
    The reverse direction is obvious.
\end{proof}

\begin{proposition} \label{prop:extension-framing-eq-F-bundle}
    Let $I$ and $J$ be  finite sets, and $R$ be a $\bbk$-algebra without zero divisors equipped with a fixed basis.
    Let $\cF\rightarrow \cF'$ be an unfolding of $\bbk$-linear equivariant F-bundles.
    Then any framing on $\cF$ extends uniquely to a framing on $\cF'$.
\end{proposition}

\begin{proof}
    Uniqueness follows from the uniqueness of extension of framing for $(\cH_R,\nabla_R)$, together with \cref{remark:morphism-eq-F-bundle}.
    
    We now prove the existence part. Assume $\cF$ admits a framing and \[(\beta,\beta_R)\colon \cF = \{(\cH,\nabla),(\cH_R,\nabla_R),\alpha\}/\bbk\dbb{t_I}\longrightarrow \cF'= \lbrace (\cH',\nabla'), (\cH_R',\nabla_R'),\alpha'\rbrace /\bbk\dbb{\bt_J}\] is an unfolding. By \cref{remark:unfolding-framing}, 
    the framing for $(\cH_R,\nabla_R)$ produces a unique framing on $(\cH_R',\nabla_R')$. By \cref{lifting lemma}(2), this framing induces a framing on $(\tcH_R',\tnabla_R')$, thus a framing on the (T)-structure $(\cH',\nabla')_0$ under $\alpha'$.
    By construction, under $\beta\vert_{\bt_I=0}$, the framing constructed on $(\cH',\nabla')$ coincides with the initial framing of $(\cH,\nabla)$.
    We conclude from \cref{lemma:framing-T-structure-to-F-bundle} that it is a framing of F-bundle. This concludes the proof.
\end{proof}

\begin{theorem}[Unfolding of equivariant F-bundles]\label{thm:max-unfolding-arbitrary-rank}
    Let $\cF = \lbrace (\cH ,\nabla ) , (\cH_R,\nabla_R) ,\alpha \rbrace$ be an equivariant F-bundle over $\bbk\dbb{\bt_I}$, and fix $v\in \cH_R / (t_I,u)\cH_R$.
    \begin{enumerate}[wide]
        \item If $v$ satisfies (IC), (GC) and $\coker\mu_v$ is free, then $\cF$ admits a maximal unfolding with cyclic vector induced from $v$.
        \item If $v$ satisfies (GC'), then any two maximal unfoldings of $\cF$ with cyclic vector induced from $v$ are isomorphic under a unique isomorphism.
    \end{enumerate}
    Furthermore, any framing for $\cF$ induces a unique framing on a maximal unfolding.
\end{theorem}

\begin{proof}
    We prove (1). For the $R$-linear (T)-structures, there exists a maximal unfolding by \cref{thm:hertling-manin-T-structure}
    \[\beta_R\colon (\cH_R,\nabla_R)/R\dbb{t_I}\longrightarrow (\cH_R',\nabla_R')/R\dbb{t_J}.\]
    By functoriality, we obtain an unfolding of $\bbk$-linear (T)-structures \[\tbeta_R\circ\alpha\colon (\cH,\nabla)_0/\bbk\dbb{\bt_I} \longrightarrow (\tcH_R',\tnabla_R')/\bbk\dbb{\bt_J}.\] 
    By \cref{lemma:weak-framing-t-directions}, the $R$-linear (T)-structures admit framings.
    Those framings induce framings on the $\bbk$-linear (T)-structures by \cref{lifting lemma}(2).
    Hence we can apply \cref{lemma:extend u-direction from T-structure} to define an F-bundle structure $(\tcH_R',\tnabla_R'^F)$ extending the $\bbk$-linear (T)-structure $(\tcH_R',\tnabla_R')$, such that $\tbeta_R\circ\alpha$ becomes an unfolding of F-bundles.
    Then  $\lbrace (\tcH_R',\tnabla_R'^F),(\cH_R',\nabla_R'),\id\rbrace$ is an equivariant F-bundle and $(\tbeta_R\circ\alpha,\beta_R)$ is a maximal unfolding of $\cF$ with cyclic vector $v$.
    
    We now prove (2).
    For $k=1,2$, let \[(\beta_k ,\beta_{R,k})\colon \lbrace (\cH,\nabla ) , (\cH_R,\nabla_R) , \alpha\rbrace \rightarrow \lbrace (\cH_k,\nabla_k)  , (\cH_{R,k},\nabla_{R,k}) , \alpha_k\rbrace\] be two maximal unfoldings of equivariant F-bundles, with cyclic vectors $v_k\in \cH_{R,k}/(t_J , u)\cH_{R,k}$ induced from $v$.
    By \cref{thm:hertling-manin-T-structure}, there exists a unique isomorphism of $R$-linear (T)-structures \[\iso_R\colon (\cH_{R,1} ,\nabla_{R,1} ) \rightarrow (\cH_{R,2} ,\nabla_{R,2} )\] such that $\beta_{R,2} = \iso_R\circ\beta_{R,1}$. 
    This induces an isomorphism for the underlying $\bbk$-linear (T)-structures \[\iso\coloneqq \alpha_{2}^{-1}\circ \widetilde{\iso_R}\circ\alpha_1\colon (\cH_1,\nabla_1)\rightarrow (\cH_2,\nabla_2),\] and it satisfies $\beta_2 = \beta_1\circ\iso$. It suffices to show that $\iso$ is compatible with the $u$-direction. Since $\beta_k$ are unfoldings of F-bundles, they restrict to isomorphisms of F-bundles at $\bt_J = 0$.
    Hence $\iso$ is compatible with the $u$-direction at $\bt_J= 0$. Since the $\bbk$-linear (T)-structures come from finite $R$-linear (T)-structures, they admit framings.
    Then \cref{lemma:lift-T-structure-morphism-to-F-bundle}(2) implies that $\iso$ is an isomorphism of F-bundles.
    We conclude that $(\iso,\iso_R)$ is an isomorphism of equivariant F-bundles compatible with the unfoldings.
    This isomorphism is unique, since $(\iso , \iso_R)$ is uniquely determined by $\iso_R$.
    (2) is proved.  

    The last statement is a special case of \cref{prop:extension-framing-eq-F-bundle}.
    The theorem is proved.
\end{proof}

\section{Application to mirror symmetry of flag varieties}

In this section, we apply our equivariant unfolding theorem to obtain the big $D$-module mirror symmetry for flag varieties $G/P$ of general Lie type (\cref{thm:big-mirror-symmetry-flag}). 

We start with the $\bbk$-linear F-bundles given by the equivariant small quantum $D$-module for $G/P$ on the A-side, and another one by the equivariant Gauss-Manin system with respect to Rietsch's superpotential on the B-side (see \cite{Riet}).
Note that both F-bundles are of infinite rank, as the equivariant parameters are not yet included in the base ring.
Moreover, their $R$-linear (T)-structure lifts coincide with the $D$-module structures defined in \cite{chow2024dmodulemirrorconjectureflag}, and are thus isomorphic to each other as shown therein.  
We will construct a suitable unfolding on the B-side, and apply our equivariant unfolding theorem to deduce the isomorphism between the unfoldings on both sides.
We remark that in general, the classical cohomology of $G/P$ is not generated by the divisor classes and the small quantum cohomology is not semisimple, so that neither the unfolding in \cite{Hertling_Unfoldings_of_meromorphic_connections} nor the semisimple reconstruction in \cite{Teleman_2011} is directly applicable.

\subsection{\texorpdfstring{Equivariant F-bundles for $G/P$}{Equivariant F-bundles for G/P}}

\subsubsection{Equivariant quantum cohomology ring of $G/P$}

Let $G$ be a simply-connected complex simple Lie group, and $P$ be a  parabolic subgroup of $G$ containing a Borel subgroup $B\subset G$. 
Let $B_-$ denote the opposite Borel subgroup, and then $T\coloneqq B\cap B_-$ is a maximal torus of $G$.  
Let $\Delta=\{\alpha_1, \dots, \alpha_n\}$ be a basis of simple roots, and $\{\omega_1, \cdots, \omega_n\}$ be the fundamental weights. 
The Weyl group $W\coloneqq N_G(T)/T$ is generated by simple reflections $s_i\coloneqq s_{\alpha_i}$.
The Weyl subgroup $W_P$ of $P$ is generated by the simple reflections $s_\alpha$ with $\alpha\in \Delta_P\coloneqq\{\alpha_i\in \Delta \mid s_i P\subset P\}$. 
Let $\ell: W\to \mathbb{Z}_{\geq 0}$ denote the standard length function, and $w_0$ (resp.\ $w_P$) denote the longest element in $W$ (resp.\ $W_P$). 
Denote by $W^P\subset W$ the subset of minimal length representative of the cosets $W/W_P$.

The flag variety $X\coloneqq G/P$ is a Fano manifold. 
It parametrizes  partial flags (resp.\ isotropic partial flags) in a complex vector space when $G$ is of type A (resp.\ B, C, D). 
For each $w\in W^P$, there are Schubert varieties $X^w\coloneqq\overline{BwP/P}$ (resp.\ $X_w\coloneqq\overline{B_-wP/P}$) of (co)dimension $\ell(w)$ inside $X$.     
We have 
\[H^{\ast} (X, \mathbb{Z})=\bigoplus_{w\in W^P} \mathbb{Z}\PD ([X_w]) ,\]
where $\PD (\cdot )$ denotes the Poincar\'{e} dual,
 and 
 \[H_2(X, \mathbb{Z})=\bigoplus_{\alpha\in \Delta\setminus \Delta_P} \mathbb{Z}[X^{s_\alpha}] .\]
For each $w\in W^P$, the Schubert variety $X_w$ (resp.\ $X^w$) is invariant under the natural  $T$-action on   $X$, so that it defines a fundamental class in the $T$-equivariant Borel-Moore homology.
This class is identified with a $T$-equivariant cohomology class in $H^{2\ell(w)}_T(X, \bbC )$ (resp.\ $H^{2(\dim X-\ell(w))}_T(X,\bbC)$) denoted as $\sigma_w$ (resp.\ $\sigma^w$). 
The fundamental weights produce equivariant parameters for the $T$-action which we denote by $\lambda = (\lambda_1,\dots, \lambda_n)$.
We have identifications
\begin{align}
    H_T^{\ast}({\rm pt},\bbC) &=\bbC[\lambda_1, \ldots, \lambda_n]\eqqcolon\bbC[\lambda] ,\\
    H^{\ast}_T(X,\bbC ) &=\bigoplus_{w\in W^P}\bbC[\lambda]\sigma_w. \label{eq: R linear basis}
\end{align}
To be more precise, we view $\omega_i$ as a character in $\Hom(T, \bbC^{\ast})$, and denote by $\bbC_{-\omega_i}$  the one-dimensional representation of $T$ viewed a vector bundle over a point. Then we take  $\lambda_i\coloneqq c_1^T(\bbC_{-\omega_i})$ and consequently we have $\lambda_i=-\omega_i$.
We denote by $(\cdot ,\cdot )$ the equivariant Poincar\'{e} pairing on $H_T^{\ast} (X ,\bbC )$. 
The $\bbC[\lambda]$-bases $\{\sigma^w\}_w$ and $\{\sigma_w\}_w$ are dual with respect to the Poincar\'{e} pairing, i.e.\ $(\sigma_u , \sigma^v) = \delta_{u,v}$.
In the following, we denote by $\bbC (\lambda )$ the fraction field of $\bbC[\lambda ] = \bbC[\lambda_1,\dots, \lambda_n]$.

\begin{lemma}[\protect{\cite[Lemma 5.11]{BCMP}}]\label{lemma:G/PH2}
The localized equivariant cohomology of $X$, $H^*_T(X)\otimes_{\bbC[\lambda]} \bbC(\lambda)$ is generated by the element $\sum_{\alpha\in \Delta\setminus\Delta_P}\sigma_{s_\alpha}$ as a $\bbC(\lambda)$-algebra.
\end{lemma}
 \begin{remark}
     The above lemma shows that $H^{\ast}_T(X,\bbC)$ is generated by $H_T^2(X,\bbC)$ after localization. 
     This also follows from \cite[Lemma 4.1.3]{ciocan-fontanine_abeliannonabelian_2008}, and can be generalized to any smooth projective variety admitting a torus action with finitely many attractive torus-fixed points by  \cite[Lemma 1]{ACT-finitenessQK}. 
 \end{remark}

Let $\overline{\mathcal{M}}_{0, m}(X, d)$ denote 
the moduli space of $m$-pointed stable maps to $X$  of genus zero and  degree $d\in H_2(X, \mathbb{Z})$,
and   ${\rm ev}_i\colon \overline{\mathcal{M}}_{0, m}(X, d)\to X$ denote the $i$-th  $T$-equivariant  evaluation map. 
The moduli space $\overline{\mathcal{M}}_{0,m} (X,d)$ carries a $T$-action, and has a $T$-equivariant virtual fundamental class $[\overline{\cM}_{0,m} (X,d)]^{\vir}$.
For $\gamma_1, \dots, \gamma_m\in H^{\ast}_T(X,\bbC)$, we have the genus-zero, 
$m$-point equivariant Gromov-Witten
invariant
\begin{equation}
    \langle \gamma_1, \dots, \gamma_m\rangle_{d}\coloneqq\int_{[\overline{\mathcal{M}}_{0,m}(X, d)]^{\vir}} {\rm ev}_1^*(\gamma_1)\cup\dots \cup {\rm ev}_m^*(\gamma_m) \in \bbC[\lambda].
\end{equation} 

We introduce the necessary choices of bases, and associated coordinates, in order to define the equivariant big quantum cohomology ring of $X$.
Write $\Delta\setminus \Delta_P=\{\alpha_{i_1}, \dots, \alpha_{i_r}\}$ and 
$W^P=\{v_1, \cdots, v_N\}$ with $v_j=s_{{i_j}}$ for $1\leq j\leq r$.
We introduce Novikov variables $q= (q_1,\dots, q_r )$ corresponding to the basis $\lbrace [X^{s_{\alpha}}] \mid \alpha\in\Delta\setminus\Delta_P\rbrace$ of $H_2 (X,\bbZ )$. 
 For $d\in H_2(X, \mathbb{Z})$, we have $d=\sum_j d_j[X^{s_{ {i_j}}}]$ and denote $q^d:=\prod_{j=1}^rq_j^{d_j}$. We use $\{\tau_i\}$ for the $\bbC[\lambda]$-linear   coordinates of $H^*_T(X)$, whose elements are of the form $\alpha=\sum_{i=1}^N\tau_i \sigma_{v_i}$.

As a module, the equivariant big quantum cohomology ring is
\[\QH_T^{\ast, \rm big}(X) \coloneqq H_T^{\ast}(X,\bbC )\otimes_{\mathbb{C}}\mathbb{C}[q]\dbb{\tau}. \] 
It encodes all genus zero Gromov-Witten invariants in the quantum product $\star_{\tau}^{\bigquantum}$,  defined by  
   \begin{align*}\label{eq:quantum_product}
\sigma_v \star_{\tau}^{\bigquantum}\sigma_w &= \sum_{\eta\in W^P} \sum_{m\ge 0}\sum_{i_1,\dots,i_m} \sum_{d\in H_2(X, \mathbb{Z})}   \frac{\tau_{i_1}\cdots \tau_{i_m}}{m!}  \braket{\sigma_v, \sigma_w, \sigma^\eta, \sigma_{v_{i_1}}, \cdots, \sigma_{v_{i_m}}}_d q^d  \sigma_\eta.
\end{align*}
Here  the coefficient of $\tau_{i_1}\cdots \tau_{i_m}$ is   indeed a polynomial in $q$ since $X$ is Fano.

 Denote $\tilde q_j:=q_je^{\tau_j}$ and $\tilde q^d:=\prod_j \tilde q_j^{d_j}$ . Letting $\tau_i=0$ for all $i>r$ and using the divisor axiom for Gromov-Witten invariants, we obtain the equivariant small quantum cohomology ring 
\[\QH^*_T(X )=H^*_T(X,\bbC )\otimes_{\mathbb{C}} \bbC[\tilde q] \quad\mbox{with}\quad \sigma_v \star  \sigma_w = \sum\nolimits_{\eta, d}   \braket{\sigma_v, \sigma_w, \sigma^\eta}_d  \tilde q^d \sigma_\eta .
\] 
The next lemma follows directly from \cref{lemma:G/PH2} and \cite[Lemma 2.1]{SiTi}.
\begin{lemma}\label{lemma:QHG/PH2}
The localized equivariant small quantum cohomology of $X$,    $\QH^*_T(X)\otimes_{\bbC[\lambda]} \bbC(\lambda)$ is generated by $\{\sigma_{s_\alpha}\mid \alpha\in \Delta\setminus\Delta_P\}$ as a $\bbC(\lambda)[\tilde q]$-algebra.
\end{lemma} 
\begin{remark}
    By further taking the nonequivariant limit $\lambda=0$, we obtain the small quantum cohomology $\QH^*(X)$, which could be non-semisimple. 
    For instance for $G$ of type $C_n$ and $\Delta_P=\Delta\setminus \{\alpha_2\}$, we obtain the isotropic Grassmannian
 $\SG(2, 2n)=\{V\leq \mathbb{C}^{2n}\mid \dim V=2, \Omega(V, V)=0\}$, where $\Omega$ is a symplectic form on $\mathbb{C}^{2n}$.
 It is shown in \cite{ChPe} that $\QH^{\ast}(\SG(2, 2n))$ is not semisimple. 
 It is easy to see that  $\QH^{\ast}(\SG(2, 2n))$ is not generated by $H^2(\SG(2, 2n),\bbC)$ either. 
\end{remark}
 
\subsubsection{Equivariant F-bundle structures for G/P}\label{subsection: Equivariant F-bundle structures of G/P}
We recall that $\tau=(\tau_1, \cdots, \tau_N)$ are the $\bbC[\lambda]$ coordinates of $H^*_T(X)$ dual to the standard basis we chose, $q=(q_1, \cdots, q_r)$ are the Novikov variables and $\lambda=(\lambda_1, \cdots, \lambda_n)$ are the equivariant variables. For $k= (k_1,\dots ,k_n)\in\bbN^n$, we set $\lambda^k\coloneqq \prod_{i=1}^n \lambda_i^{k_i}$ and $\vert k\vert \coloneqq \sum_{i=1}^n k_i$. It is expected but remains unsolved in general that the big quantum cohomology is convergent around $\tau=0$. 
Therefore we work on the formal neighborhood of $\tau = 0$. 

Let $\bbk\coloneqq \mathbb{C}(q)$ be the fraction field of $\bbC[q]$, and let  $R\coloneqq\bbk[\lambda]$. 
We fix the $\bbk$-basis $\blambda = (\lambda^k , k\in\bbN^n)$ of $R$.
We obtain $\bbk$-linear coordinates $\btau  = \lbrace \tau_{i,k}  ,1\leq i\leq N , k\in\bbN^n\rbrace$ on $H_T^{\ast} (X,\bbk)$ associated to the $\bbk$-basis $(\sigma_{v_i}\lambda^k, 1\leq i\leq N, k\in \bbN^n)$. There is a continuous morphism of $R$-algebras 
\[\psi_{\blambda} : R\dbb{\tau}\rightarrow R\dbb{\btau}, \tau_i \mapsto \sum_{k\in \bbN^n}\lambda^k\tau_{i,k}.\]

We define a $\bbk$-linear equivariant F-bundle  equivariant F-bundle 
\[\cF^{A,\bigquantum} \coloneqq \{(\cH^{A, {\rm big}}, \nabla^{A, {\rm big}}), (\cH_R^{A, {\rm big}}, \nabla_R^{A, {\rm big}}), \alpha\} / \bbk\dbb{\btau}. \] associated to the equivariant big quantum cohomology as follows.
The $R$-linear (T)-structure $(\cH_R^{A, {\rm big}}, \nabla_R^{A, {\rm big}})$ is given by the $R\dbb{\tau, u}$-module 
\begin{align*}
    \cH_R^{A,{\rm big}}& =H^{\ast}_T(X,\bbk)\otimes_R   R\dbb{\tau, u},\\
     \nabla_{R, {\partial_{\tau_{j}}}}^{A, {\rm big}}&=\partial_{\tau_{j}}+ u^{-1}\big((\sigma_{v_j}+\lambda_{i_j}) \star_{\tau}^{\bigquantum}
\big) ,
\end{align*}
where $1\leq j\leq N$ and we set $\lambda_{i_j}=0 $ for $j>r$. Here $i_{j}$ are the index of  $\Delta\setminus \Delta_P=\{\alpha_{i_1}, \dots, \alpha_{i_r}\}$.
The $\bbk$-linear F-bundle $(\cH^{A,\mathrm{big}}, \nabla^{A,\mathrm{big}})$ has underlying $\bbk\llbracket \btau, u \rrbracket$-module
\[
\cH^{A,\mathrm{big}} = H^*_T(X, \bbk) \otimes_{\bbk} \bbk\llbracket \btau, u \rrbracket,
\]
and the connection $\nabla^{A,\mathrm{big}}$ is specified by:
\[
\begin{aligned}
\nabla_{\partial_{\tau_{j,k}}}^{A,\mathrm{big}} &= \partial_{\tau_{j,k}} + u^{-1} \big( \lambda^k (\sigma_{v_j} + \lambda_{i_j}) \star_{\tau}^{\mathrm{big}} \big), \\
\nabla_{u\partial_u}^{A,\mathrm{big}} &= \Gr^{A,\bigquantum} - \nabla_{E^{\mathrm{big}}_A}^{A,\mathrm{big}} .
\end{aligned}
\]
Here, 
\[\Gr^{A,\bigquantum}=u\partial_u + E^{\mathrm{big}}_A+ \mu_{A},\]
where $\mu_A$ is the $\bbk\dbb{\btau,y}$-linear grading operator on the fiber $H^*_T(X, \bbk)$ linear defined by
\[\mu_A(\lambda^k \sigma_{v_i}) = \big( \ell(v_i) + |k| \big) \lambda^k \sigma_{v_i} ,
\] 
and $E^{\mathrm{big}}_A$ is the Euler vector field measuring degree on the base $\bbk\dbb{\btau}$, given by
    \[
    E^{\mathrm{big}}_A = \sum_{1 \leq j \leq r} \frac{\deg(q_j)}{2} \partial_{\tau_{j,0}} + \sum_{\substack{1 \leq j \leq N \\ k \in \mathbb{N}^n}} (1 - \ell(v_j) - |k|) \tau_{j,k} \partial_{\tau_{j,k}},
    \]
    where $\ell(v_j)=1$ for $1\leq j\leq r$ and the degree $\deg(q_j)$ is defined as
    \[
    \deg(q_j) \coloneqq 2 \int_{[X^{s_{i_j}}]} c_1(T_{G/P}).
    \]
Under the change of variables $R\dbb{\tau}\rightarrow R\dbb{\btau}$, $\tau_i\mapsto \sum_{k\in\bbN^n} \lambda^k \tau_{i,k}$, the data $\lbrace (\cH_R^{A,\bigquantum} , \nabla_R^{A,\bigquantum}) , \id\rbrace$ provides an $R$-linear lift of the underlying (T)-structure $(\cH^{A,\bigquantum} , \nabla^{A,\bigquantum})_0$.
We obtain the $A$-model big equivariant F-bundle
\[\cF^{A,\bigquantum} = \{(\cH^{A, {\rm big}}, \nabla^{A, {\rm big}}), (\cH_R^{A, {\rm big}}, \nabla_R^{A, {\rm big}}), \id\} / \bbk\dbb{\btau}.\]

\begin{remark}

  For $1\leq j\leq r$, consider the line bundle $L_j=G \times_P \bbC_{-\omega_{i_j}}$ over $G/P$. 
  Since $c_1^T(L_j)= \sigma_{s_{i_j}}-\omega_{i_j}=\sigma_{v_j}+\lambda_{i_j}$,
  we can write $\nabla_{R,\partial_{\tau_j}}^{A,\bigquantum} = \partial_{\tau_j} + u^{-1}c_1^T (L_j)\star_{\tau}^{\bigquantum}$. For $j>r, \nabla_{R,\partial_{\tau_j}}^{A,\bigquantum}$ are not weighted.
\end{remark}

\begin{remark}
\label{remark:flatness-A-model}
The flatness   of $\nabla^{A, {\rm big}}$ in the $u$-direction follows from a similar argument to that in   \cite[Section 3.2]{Coates_Hodge-theoretic_mirror_symmetry_for_toric_stacks} 
.\end{remark}

By restricting to the small locus $\tau_j = 0$ for $(\cH^{A, {\rm big}}, \nabla^{A, {\rm big}})$, and $\tau_{j,k} = 0$ for $(\cH_R^{A, {\rm big}}, \nabla_R^{A, {\rm big}})$ when $j>r$, we obtain the $A$-model small equivariant F-bundle
\[\cF^{A}\coloneqq \lbrace (\cH^{A}, \nabla^{A}), (\cH_R^{A}, \nabla_R^{A}), \id \rbrace / \bbk\dbb{\btau_{\leq r}} ,\]
where $\btau_{\leq r} = \lbrace \tau_{i,k} , 1\leq i\leq r, k\in\bbN^n\rbrace$ parametrizes the $\bbk$-linear F-bundle, and $\tau_{\leq r} = \lbrace \tau_i ,1\leq i\leq r\rbrace$ parametrizes the $R$-linear (T)-structure.

The quotient maps $\bbk\dbb{\btau}\rightarrow \bbk\dbb{\btau_{\leq r}}$ and $R\dbb{\tau}\rightarrow R\dbb{\tau_{\leq r}}$ together with the natural identification of the fibers produce an unfolding of the equivariant F-bundle $\iota\colon \cF^{A}\rightarrow \cF^{A,\bigquantum}$.

\begin{proposition}\label{prop: A model maximal}
The morphism $\iota\colon \cF^{A}\rightarrow \cF^{A,\bigquantum}$ is a maximal unfolding of $\bbk$-linear equivariant F-bundles, with cyclic vector given by $1\in H_T^{\ast}(X,\bbk)$.
\end{proposition}
\begin{proof}
  We have already proven the morphism $\iota:\cH^A\rightarrow \cF^{A,\bigquantum}$ is an unfolding, and it remains to check that $(\cH^A_R,\nabla^A_R)$ is maximal. 
  We take the cyclic vector $v=1\in \cH_{R}^{A,\bigquantum}|_{\tau=u=0}=H^*_T(X,\bbk)$. 
  The $R$-linear evaluation map is the 
  \begin{align*}
      \mu_{v=1}\colon \bigoplus_{j=1}^N R\partial_{\tau_j} &\longrightarrow H_T^{\ast} (X,\bbk)\\
      \partial_{\tau_j} &\longmapsto \nabla_{u\partial_{\tau_j}}|_{\tau=u=0}(1) = \sigma_{v_j}+\lambda_{i_j}.
  \end{align*}
  Since $\{\sigma_{v_i}\}_i$ is an $R=\bbk[\lambda]$ basis of $H^\ast_T(X,\bbk)$ by equation (\ref{eq: R linear basis}), $\mu_{v=1}$ is an $R$-isomorphism and we conclude that $\iota$ is maximal unfolding. 
\end{proof}

\subsection{\texorpdfstring{The small $D$-module mirror symmetry for $G/P$}{The small D-module mirror symmetry for G/P}} \label{subsec:small-B-model}
In this section, we review the $B$-side of mirror symmetry for for $G/P$ as in \cite{Riet}, construct a $R$-linear (T)-structure $(\cH_R^B,\nabla_R^B)$ from the Gauss-Manin connection, and state the small mirror symmetry as in \cite{chow2024dmodulemirrorconjectureflag}.

\subsubsection{Small $D$-module mirror symmetry}
Let $G^\vee$ be the Langlands dual group of $G$, and $T^\vee$, $B^\vee$, $P^\vee$ be the Langlands dual of $T, B, P$ respectively.
Rietsch's equivariant mirror superpotential is a triple $(X^{\vee}_{P}, \mathcal{W}, p)$. Here $X^\vee_P$ is a subvariety of $G^\vee\times Z$ isomorphic to 
\[\big((G^\vee/P^\vee)\setminus -K_{G^\vee/P^\vee} \big)\times \mbox{Spec } \mathbb{C}[\tilde q_{i}^{\pm 1}|\alpha_i\in \Delta\backslash \Delta_P ],\]
where $Z$ is the center of the Levi subgroup of $P^\vee$, $-K_{G^\vee/P^\vee}$ is the anti-canonical divisor of the dual partial flag variety ${G^\vee/P^\vee}$ given in \cite{KLS}. The holomorphic function
$\mathcal{W}: X^\vee_P\to \mathbb{C}$ is the non-equivariant mirror superpotential of $G/P$, and $p: X^\vee_P\to T^\vee$ is a morphism which gives information on the equivariant part of Rietsch's original mirror superpotential $\mathcal{W}+\ln \phi( ; h)$ (see \cite{Riet}).

  Denote by $\Omega^i(X^\vee_P/Z)$ the space of holomorphic $i$-forms over $X^\vee_P$ with respect to \[Z\cong\mbox{Spec } \mathbb{C}\big[\tilde q_{i}^{\pm 1}\ \big|\ \alpha_i\in \Delta\backslash \Delta_P \big]\] via the aforementioned isomorphism. Identify the Lie algebra $\mathfrak{t}^\vee=\mbox{Lie}(T^\vee)$ with $\mathfrak{t}^*$.
Let $\{(\lambda_i)^{*}\}\subset (\mathfrak{t}^\vee)^*$ be the dual base of $\{\lambda_i\}\subset \mathfrak{t}^\vee$, and 
${\rm mc}_{T^\vee}\in \Omega^1(T^\vee; \mathfrak{t}^\vee)$ denote the
Maurer-Cartan form of $T^\vee$.

In \cite{chow2024dmodulemirrorconjectureflag}, the $B$-model $D$-module $\big(G_0(X_P^{\vee}, \mathcal{W}, p),\nabla\big)$ consists of a $\bbC[\lambda,u]
[\tilde{q}_i^{\pm 1}]$-module defined by
\begin{align*}
     G_0(X_P^{\vee}, \mathcal{W}, p)&= \mbox{coker} \big(\bbC[\lambda, u]\otimes_\bbC \Omega^{{\rm top}-1}(X^\vee_P/Z)\overset{\partial}{\longrightarrow} \bbC[\lambda,u]\otimes_\bbC \Omega^{\rm top}(X^\vee_P/Z)\big),\\ 
     \partial & = 1\otimes \big(ud + d\mathcal{W}\wedge  -\sum_{j=1}^n  \lambda_j  (p^*\langle (\lambda_j)^{*}, {\rm mc}_{T^\vee}\rangle)\wedge \big),
\end{align*}
It is equipped with a meromorphic connection having a logarithmic pole in the $\tilde{q_i}$-direction
\[\nabla_{  \partial_{\tilde{q}_i}}([\omega]) = 
\big[\mathcal{L}_{\partial_{\tilde{q}_{i}}} (\omega)+ u^{-1}\frac{\partial \mathcal{W}}{\partial \tilde{q}_{i} } \omega  -u^{-1}\sum_{j=1}^n \lambda_j \big(\iota_{\partial_{\tilde{q}_{i}}} p^{\ast} \big\langle (\lambda_j)^{\ast}, {\rm mc}_{T^{\vee}}  \big\rangle\big)\omega \big].
\]
Here $p^{\ast} \big\langle (\lambda_j)^{\ast}, {\rm mc}_{T^{\vee}}  \big\rangle \in \Omega^1(X^{\vee}_{P}/Z)$, and $\partial, \nabla$ are linear on $\mathbb{C}[\lambda, u]$.

\remark{For any $\omega\in \bbC[\lambda,u]\otimes \Omega^{\rm top}(X^\vee_P/Z)$, $\omega=g \omega_{0}$ for some $g\in \mathcal{O}(X^{\vee}_{P})$ and $\omega_{0}\in \Omega^{\rm top}( (G^\vee/P^\vee)\setminus -K_{G^\vee/P^\vee} )$ }. We have the Lie derivative $ \mathcal{L}_{\partial_{ {\tilde{q}_{i}}}}(\omega)=\frac{\partial g}{\partial \tilde{q_{i}}}\omega_{0} .$

The small quantum $D$-module mirror symmetry holds for $G/P$ in the following sense.
\begin{proposition}[\protect{\cite[Theorem 1.2]{chow2024dmodulemirrorconjectureflag}}]\label{prop: iso_small_Dmodule current version}
    There exists a unique $\bbC[\lambda,u][\tilde{q_i}^{\pm 1}]$-linear map
    \begin{equation}\label{eq:chow-mirror-map}
        \Phi_{\rm mir}: G_0(X^{\vee}_{P}, \mathcal{W},  p)\longrightarrow \QH_{T}^{\ast}(G/P)[u, \tilde q_1^{-1},\cdots, \tilde q_r^{-1}] 
    \end{equation}
    satisfying the following:
    \begin{enumerate}[wide]
        \item $\Phi_{\rm mir}$ is bijective, and preserves the connection, 
        \item $\Phi_{\rm mir}([\Omega])=1$, where $\Omega$ is the unique (up to sign) volumn form in $\Omega^\top (X_P^\vee/Z)$, whose restrictions to every torus chart of $(G^\vee/P^\vee)\setminus -K_{G^\vee/P^\vee}$ is equal to the standard volumn form $\pm dz_1\wedge \dots \wedge dz_K/z_1\dots z_K,$
        \item  at the semi-classical limit, we have  a ring isomorphism
        \begin{equation} \label{eq:jacobi-isomorphism}
            \Phi_{\rm mir}^{u=0}: {\rm Jac}(X^{\vee}_{P}, \mathcal{W},  p)\overset{\cong}{\longrightarrow} \QH_{T}^{\ast}(G/P)[\tilde q_1^{-1},\cdots, \tilde q_r^{-1}],
        \end{equation}
        \item $\Phi_{\rm mir}$  intertwines the shift operators (see \cite[Sections 3.3 and 4.5]{chow2024dmodulemirrorconjectureflag}), and
        \item $\Phi_{\mir}$ preserves the $\mathbb{Z}$-grading.
    \end{enumerate}
\end{proposition}
In \eqref{eq:jacobi-isomorphism}, ${\rm Jac}(X^{\vee}_{P}, \mathcal{W},  p)$ denotes the Jacobi ring,  which  is the coordinate ring of the scheme-theoretic zero locus of certain relative 1-forms in $\Omega^1(X^\vee_P\times \mathfrak{t}/Z\times \mathfrak{t})$ (see \cite[Definition 4.9]{chow2024dmodulemirrorconjectureflag} for more details).
It corresponds to setting $u=0$ in the $B$-model $D$-module. 

We remark that the above isomorphism is a bit implicit. Below we provide an example with explicit isomorphism of small quantum $D$-modules from \cite{MaRi}.

\begin{example}
For $G/P=\Gr(3, 5)=\{V\leq \mathbb{C}^5\mid \dim V=3\}$, the Langlands dual flag variety  $G^\vee/P^\vee$ is the Grassmannian $\Gr(2, 5) \hookrightarrow \mathbb{P}^9$, whose image   is defined by the Pl\"ucker  relations $p_{a_1a_2}p_{a_3a_4}-p_{a_1a_3}p_{a_2a_4}+p_{a_1a_4}p_{a_2a_3}=0$ for $1\leq a_1<a_2<a_3<a_4\leq 5$. 
   In this case,  $-K_{G^\vee/P^\vee}=\{p_{12}p_{23}p_{34}p_{45}p_{15}=0\}$.
   The $\mathcal{W}$-part of Rietsch's equivariant superpotential 
   is given by 
   $$\mathcal{W}=\frac{p_{13}}{p_{12}}+\frac{p_{24}}{p_{23}}+\frac{p_{35}}{p_{34}}+\tilde q \frac{p_{14}}{p_{45}}+\frac{p_{25}}{p_{15}}.$$
   The degree of the inhomogeneous coordinate $\theta_{ij}=\frac{p_{ij}}{p_{12}}$ is equal to $2(i+j-3)$.
   The volume form
   $\Omega=\frac{d\theta_{13}d\theta_{14}d\theta_{15}d\theta_{23}d\theta_{24}d\theta_{25}}{\theta_{23}\theta_{34}\theta_{45}\theta_{15}}$ is of degree $0$.
   For $1\leq i<j\leq 5$, 
   $\Phi_{\rm mir}([\theta_{ij} \Omega])=\sigma_w$
   with $w\in S_5$ the unique permutation satisfying   $w(4)=6-j$, $w(5)=6-i$ and 
    $w(1)<w(2)<w(3)$.
\end{example}

\subsubsection{Equivariant F-bundles formulation}
In our setting of (T)-structures, we need to replace the logarithmic $\tilde{q}_i$-directions with a regular meromorphic connection in $y_i$-directions. This is achieved by the $D$-module inverse image under
\begin{align*}
\psi_1\colon \bbC[\lambda,u][\tilde{q_i}^{\pm 1}]&\longrightarrow \bbC[\lambda,u][q_i^{\pm 1}]\dbb{y_{\leq r}}\\
    \tilde{q}_i &\longmapsto q_i e^{y_i},
\end{align*}
where $y_{\leq r} = \lbrace y_i, 1\leq i\leq r\rbrace$. For the purpose of applying our reconstruction theorem to obtain big mirror symmetry, we need to further take the fraction field of $q$ and formalize $u$. 
So we compose $\psi_1$ with the following base change
\[\psi_2\colon \bbC[\lambda,u][q_i^{\pm 1}]\dbb{y_{\leq r}} \longrightarrow  \bbC[\lambda,u][q_i^{\pm 1}]\dbb{y_{\leq r}}\otimes_{\bbC[q_i^{\pm 1},u]}\bbC(q)\dbb{u}.\]
Namely we have the following, where we recall $\bbk=\bbC(q)$ and $R=\bbk[\lambda]$. 
\begin{equation}\label{eq1: change ring}
    \psi\coloneqq \psi_2\circ\psi_1\colon\bbC[\lambda,u][\tilde{q}_i^{\pm 1}]\longrightarrow R\dbb{y_{\leq r},u}.
\end{equation}

The $B$-model $R$-linear (T)-structure  $(\cH_R^{B} , \nabla_R^{B})$ is the $D$-module inverse image $\psi^*\big(\big(G_0(X_P^{\vee}, \mathcal{W}, p),\allowbreak \nabla\big)\big)$, whose underlying $R\dbb{y_{\leq r},u}$-module is
\[\cH_R^{B} = G_0 (X_P^{\vee} , \cW , p)\otimes R\dbb{y_{\leq r} , u}, \]
and the connection is given by 
\begin{align}\label{eq:Gauss-Manin-connection}
    \nabla^B_{R, \partial_{y_{i}}}([\omega])=\big[\mathcal{L}_{\partial_{y_{i}}} (\omega)+ u^{-1}\frac{\partial \mathcal{W}}{\partial y_{i} } \omega  -u^{-1}\sum_{j=1}^n \lambda_j \big(\iota_{\partial_{y_{i}}} p^{\ast} \big\langle (\lambda_j)^{\ast}, {\rm mc}_{T^{\vee}}  \big\rangle\big)\omega \big].
\end{align}

Next we define the $B$-model $\bbk$-linear F-bundle. Fix the $\bbk$-basis of $R$ given by $\blambda = (\lambda^k , k\in \bbN^n)$ and let $\by_{\leq r} \coloneqq \lbrace y_{i,k}, 1\leq i\leq r, k\in\bbN^n \rbrace$. Consider the following change of variables
\begin{align}\label{eq2: change variables}
    \psi_{\blambda}:R\dbb{y_{\leq r},u} &\longrightarrow R\dbb{\by_{\leq r},u}
    \\y_i &\longmapsto \sum_{k\in\bbN^n} \lambda^k y_{i,k}. \nonumber
\end{align}

The underlying (T)-structure 
 $(\cH^B,\nabla^B)$ of the $B$-model $\bbk$-linear F-bundle is defined as the $D$-module inverse image $\psi_{\blambda}^*(\cH_R^B,\nabla^B_R)$ and restrict scalars from $R$ to $\bbk$, as in \cref{lifting lemma}. Explicitly, the underlying $\bbk\dbb{\by_{\leq r},u}$-module is 
\[\cH^{B}= G_0 (X_P^{\vee} ,\cW ,p)\otimes  R\dbb{\by_{\leq r},u},\]
equipped with a regular meromorphic connection 
\[ \nabla^{B}_{\partial_{y_{i, k}}}([\omega ])=\big[\mathcal{L}_{\partial_{y_{i,k}}} (\omega)+  u^{-1}\frac{\partial \mathcal{W}}{\partial y_{i,k} } \omega  -u^{-1}\sum_{j=1}^n \lambda_j \big(\iota_{\partial_{y_{i, k}}} p^{*} \big\langle (\lambda_j)^{*}, {\rm mc}_{T^{\vee}}  \big\rangle\big)\omega \big],\]
where $\mathcal{W}$ is in variables ${y_{i,k}}$.
The $u$-direction is defined in \cref{eq:u-direction-small-B-model}, its definition uses the grading operator on the $B$-model, which we now define. 
We first construct a grading on $\Omega^{\top} (X_P^{\vee}/Z) \otimes \bbk\dbb{\by_{\leq r},u}$.

\begin{construction}[Grading on differential forms] \label{constr:grading-fiber-small-model}
We construct a $\bbk\dbb{\by_{\leq r},u}$-linear operator $\mu_B$ on $\Omega^{\top} (X_P^{\vee} / Z ) \otimes\bbk\dbb{\by_{\leq r},u}$, which defines a grading on differential forms. 

Recall that $\Omega^{\top} (X_P^{\vee} /Z) \otimes\bbk\dbb{\by_{\leq r} , u}$ is a rank $1$ free module over $\cO\big( X_P^{\vee} \big)\dbb{\by_{\leq r}, u}$.
The choice of $\Omega$ in \cref{prop: iso_small_Dmodule current version}(2) produces a basis of this module. 
For any differential form  $\omega = h\Omega$ with $h\in\cO\big( X_P^{\vee} \big)\dbb{\by_{\leq r}, u}$, we define
\[\mu_B ( h\Omega ) \coloneqq \frac{\deg_B (h)}{2}\Omega ,\]
where the degree operator $\deg_B$ is defined on functions in $\cO (X_P^{\vee} )$ using the $\bbG_m$-action on $X_P^{\vee}$ constructed in \cite[Lemma 4.6]{chow2024dmodulemirrorconjectureflag}, and extended by $\bbk\dbb{\by_{\leq r} ,u}$-linearity. 
Using the Jacobian isomorphism \eqref{eq:jacobi-isomorphism} and the grading operator $\mu_A$, we can describe $\deg_B$ as follows. For any local chart, we take a coordinate system $(z_{i})_{i}$ so that $z_{i}$ are homogeneous.
 For a monomial function $h\in \cO \big((G^\vee/P^\vee)\setminus -K_{G^\vee/P^\vee} \big)$, we have $\deg_B (h) = 2dh$ where $d\in \bbZ$ is given by:
\[\big(\sum_{i=1}^{r} \frac{\deg(q_{i})}{2}q_{i}\frac{\partial}{\partial q_{i}}+ \mu_{A} \big) \Phi_{\mir}^{u=0,y=0} (\overline{h})=d \Phi_{\mir}^{u=0,y=0} (\overline{h}),\]
where $\overline{h}$ denotes the image of $h$ in $\Jac (X_P^{\vee}, \cW ,p)$ and $\Phi_{\mir}^{u=0,y=0}$ is the Jacobi isomorphism \eqref{eq:jacobi-isomorphism}. It is then extended to a $\bbk\dbb{\by_{\leq r},u}$-linear operator on $\cO (X_P^{\vee} )\dbb{\by_{\leq r},u}$.
\end{construction}

We use the simple example of $X=\mathbb{CP}^{1}$ case to illustrate the definition of $\deg_{B}$.
\begin{example}[$\mu_B$ for $\bbC\bbP^1$]
For $X=\mathbb{CP}^{1}$, we have $\bbk= \mathbb{C}(q)$.
The equivariant (small) quantum cohomology $\QH_T^{\ast} (X)$ is isomorphic to $\bbk[ H,\lambda] / (H^2 - H\lambda - q)$, where $q$ has degree $4$.
The mirror $X_P^{\vee}$ is the family $\bbG_m\times \Spec \bbC [q^{\pm 1}]\rightarrow \Spec \bbC [q^{\pm 1} ]$, the superpotential is $\cW = z+ \frac{q}{z}$ and the Jacobi ring $\Jac (X_P^{\vee} , \cW, p)$ is isomorphic to $\bbk [z,z^{-1} ,\lambda ] / \big( 1 - \frac{\lambda}{z} - \frac{q}{z^2}\big)$.
The mirror isomorphism $\Phi_{\mir}^{u=0,y=0}$ at $u=0$, $y=0$ is given the morphism of $\bbk [\lambda]$-modules defined by $z\mapsto H$.

We have $\deg_B (qz^3) = q\deg_B (z^3) = 6qz^3$, where the last equality follows from the computation:
\begin{align*}
(\frac{\deg(q)}{2}q\frac{\partial}{\partial q} + \mu_{A} ) \Phi_{\mir}^{u=0,y=0} (\overline{z^{3}})&= \big(\frac{\deg(q)}{2}q\frac{\partial}{\partial q} + \mu_{A} \big)(\lambda^{2}H+qH+q\lambda)\\
&=3 (\lambda^{2}H+qH+q\lambda). 
\end{align*}
\end{example}

The $u$-direction connection of the $B$-model $\bbk$-linear F-bundle $(\cH^B,\nabla^B)$ is defined as 
\begin{align} \label{eq:u-direction-small-B-model}
 \nabla^{B}_{u\partial_u}&\coloneqq \Gr^{B} - \nabla_{E^{\rm }_{B}} , \\
 \Gr^B &\coloneqq u\partial_u + E_B + \mu_B, \nonumber
\end{align}
where $u\partial_u$ measures the degree in $u$, $\mu_B$ is the grading operator on differential forms (see \cref{constr:grading-fiber-small-model}), and $E_B$ is the Euler vector field measuring the degree of the $\by$-variables and accounting for the degree of $q$:
\[E^{\rm }_{B}\coloneqq \sum_{1\leq j\leq r}\frac{\deg(q_j)}{2}\partial_{y_{j,0}} -\sum_{1\leq j\leq r,k\in \bbN^{n}}|k| y_{j,k}\partial_{y_{j,k}}.\]
In the following proposition, we note that even though $\mu_B$ is only defined on differential forms, the total grading operator $\Gr^B$ is well-defined on equivalences classes.

\begin{proposition}\label{grading operator well defined.}
The grading operator $\Gr^{B}$ produces a well-defined operator on $\cH^B$.
\end{proposition}
\begin{proof}
Let $K$ denote the rank of $\Omega^1 (X_P^{\vee} /Z)$.
We only need to prove that for any $(K-1)$-form $\eta$, there exists a $(K-1)$-form $\eta'$ such that
\begin{align}\label{grading on n-1 form}
    \Gr^{B}\big((ud+d\mathcal{W}\wedge)\eta\big)= (ud+d\mathcal{W}\wedge)\eta'
\end{align}
Fix a torus-invariant local chart with coordinates $(z_i)_i$ and write
\[\eta=\sum_i g_{i} (z,q,\by_{\leq r},u) \iota_{\partial_{z_{i}} }\big(\bigwedge_{j=1}^{K} dz_{j}\big).\]
Let 
\[\eta' \coloneqq\sum_i \big( (u\partial_u + E_B + \deg_B)(g_{i} )+\big(1-\frac{\deg_{B}(z_{i})}{2 z_{i}}\big)g_{i}  \iota_{\partial_{z_{i}} }\big(\bigwedge_{j=1}^{K} dz_{j}\big) .\]
A direct computation shows that this choice of $\eta'$ satisfies \eqref{grading on n-1 form}.
We conclude that $\Gr^B$ descends to an operator on $\cH^B$, completing the proof. 
\end{proof}



We note the following property of $\mu_B$, which we will use in \cref{remark:flatness-B-model}.

\begin{lemma}[Leibniz rule for $\mu_B$]\label{eq: mu_b leibniz rule}
The grading operator $\mu_B$ on differential forms $\mu_{B}$ (see \cref{constr:grading-fiber-small-model}) satisfies the Leibniz rule:
\[\mu_{B}(g_{1}g_{2}\Omega  )=g_{1}\mu_{B}(g_{2}\Omega)+g_{2}\mu_{B}(g_{1}\Omega), \]
for any $g_{1},g_2\in\cO(X_P^\vee)\otimes \bbk\dbb{\by_{\leq r},u}$.
\end{lemma}
\begin{proof}
The statement follows from the facts that the Jacobian isomorphism $\Phi_{\rm mir}^{u=0, y=0}$ is a ring isomorphism, and that $(\sum_{j=1}^{r} \frac{\deg(q_{j})}{2}q_{j}\frac{\partial}{\partial q_{j}}+ \mu_{A} )$ satisfies the Leibniz rule. 
\end{proof}

By construction, the data $\lbrace (\cH_R^{B} , \nabla_R^{B} ) ,\id \rbrace$ is an $R$-linear lift of the $\bbk$-linear (T)-structure $(\cH^{B} , \nabla^{B} )_0$, and we obtain the $B$-model small equivariant F-bundle
\[\cF^B = \lbrace(\cH^{B}, \nabla^{B}), (\cH_R^{B}, \nabla_R^{B}), \id \rbrace / \bbk\dbb{\by_{\leq r}}  .\]


\begin{lemma}[\cref{prop: iso_small_Dmodule current version}(5)] \label{Chow compatible with grading}
For $h\in \cO((G^\vee/P^\vee)\backslash-K_{G^\vee\backslash P^\vee})$, we have
\[\big(u\partial_u+\sum_{i=1}^{r} \frac{\deg(q_{j})}{2}q_{j}\frac{\partial}{\partial q_{j}}+ \mu_{A} \big) \Phi_{\mir}^{y=0} (h\Omega)=\Phi_\mir^{y=0}\big((u\partial_u+\mu_B)(h\Omega) \big).\]
\end{lemma}

The following proposition essentially follows from Proposition \ref{prop: iso_small_Dmodule current version}.
Here we add the explanation in detail for completeness.

\begin{proposition}\label{small mirror} 
The $A$-model small equivariant F-bundle $\cF^A$ over $\bbk\dbb{\btau_{\leq r}}$ is isomorphic to the $B$-model small equivariant F-bundle $\cF^B$ over $\bbk\dbb{\by_{\leq r}}$ under an isomorphism  $\big((\mir_\bbk,{\Phi_{\mir,\bbk}}), (\mir,\Phi_\mir) \big)$ where $\Phi_{\mir}$ is as in \cref{prop: iso_small_Dmodule current version}, $\Phi_{\mir ,\bbk}$ is the induced $\bbk$-linear map, and $\mir_\bbk\colon y_{i, k}\mapsto \tau_{i, k}$, $\mir\colon y_{i}\mapsto \tau_{i}$ identify the variables of the equivariant F-bundles.
\end{proposition}
\begin{proof} 
Our construction of $B$-model small equivariant F-bundle \[\cF^B = \lbrace(\cH^{B}, \nabla^{B}), (\cH_R^{B}, \nabla_R^{B}), \id \rbrace / \bbk\dbb{\by_{\leq r}}\] 
consists of the (T)-structure obtained as $D$-module inverse image:
\begin{align*}
    (\cH^B_R,\nabla^B_R) &= \psi^* \big(G_0(X_P^{\vee}, \mathcal{W}, p),\nabla\big),\\
    (\cH^B,\nabla^B) &= (\psi_{\blambda}\circ \psi)^* \big(G_0(X_P^{\vee}, \mathcal{W}, p),\nabla\big),
\end{align*}
where $\psi$ and $\psi_{\blambda}$ are base change maps defined in equations \eqref{eq1: change ring} and \eqref{eq2: change variables}, together with a $u$-direction on $(\cH^B,\nabla^B)$.
\\Our construction of $A$-model small equivariant F-bundle $\cF^{A}= \lbrace (\cH^{A}, \nabla^{A}), \allowbreak (\cH_R^{A}, \nabla_R^{A}), \id \rbrace / \bbk\dbb{\btau_{\leq r}}$ in \cref{subsection: Equivariant F-bundle structures of G/P}, can also be written as (T)-structures $(\cH^A,\nabla^A)$ and $(\cH^A_R,\nabla^A_R)$ obtained as $D$-module inverse image under the same maps, together with a $u$-direction on $(\cH^A,\nabla^A)$. \\
By the functoriality of pullback, the isomorphism in Proposition \ref{prop: iso_small_Dmodule current version} induces an isomorphism
\[((\mir_\bbk,\Phi_{\mir,\bbk}),  (\mir,\Phi_{\mir}))\colon\cF^B\longrightarrow \cF^A\]
of the underlying (T)-structures. 
It suffices to prove that in our setting $\Phi_{\mir}$ is also graded, i.e.\ that $(\mir_{\bbk})_{\ast} E_B = E_A$ and
\begin{equation}\label{eq:compatibility-gradings}
    \Phi_{\mir,\bbk}\circ \Gr^{B} = \mir_\bbk^\ast \Gr^A \circ \Phi_{\mir,\bbk},
\end{equation}
where $\mir_\bbk^\ast \Gr^A$ is the grading operator on $\mir_\bbk^*\cH_A$ induced from $\Gr^A$. 
The compatibility of the Euler vector fields is clear from their definition and the fact that $\mir_{\bbk}$ identifies the variables $\by_{\leq r}$ and $\btau_{\leq r}$.
For \eqref{eq:compatibility-gradings}, we need to check that for all $f\in O(X_P^\vee)\dbb{\by_{\leq r},u}$, we have 
\[ \Phi_{\mir,\bbk}([(u\partial_{u}+E_{B}+\mu_{B})(f\Omega)])= (u\partial_{u}+E_B+\mu_{A})(\Phi_{\mir,\bbk}([f\Omega])).\]
Since both sides are linear in $q$ and satisfy the Leibniz rule for $y$ and $u$, it suffices to check the equation for $f\in \cO\big((G^\vee/P^\vee)\backslash -K_{G^\vee/P^\vee}\big)$.

In this case, the left-hand side reduces to $\Phi_{\mir,\bbk} ([\mu_{B}(f\Omega)])$.
For the right-hand side, observe that $\Phi_{\mir,\bbk} ([f\Omega] )$ is obtained by first applying the mirror map \eqref{eq:chow-mirror-map} from \cite{chow2024dmodulemirrorconjectureflag} to $[f\Omega ]$, and then pulling back under the change of variables $\tq_i \mapsto  q_i e^{\sum_k \lambda^k y_{i,k}}$.
Hence $\Phi_{\mir,\bbk} ([f\Omega] )$ is a power series in the variables $\lbrace q_ie^{\sum_k \lambda^k y_{i,k}}\rbrace_{1\leq i\leq r}$ with coefficients in $H_T^{\ast} (X,\bbC )$.
For an equivariant cohomology class $\alpha\in H_T^{\ast}(X,\bbC )$, we have 
\[q_i\frac{\partial}{\partial q_i} (q_i e^{\sum_k \lambda^k y_{i,k} }\alpha) = \frac{\partial}{\partial y_{i,0}} (q_ie^{\sum_k \lambda^k y_{i,k}} \alpha).\]
It follows that at $y=0$, the right hand side is 
\begin{align*}
    (u\partial_u + E_B + \mu_A)(\Phi_{\mir,\bbk} ([f\Omega] )) &= \big( u\partial_u + \sum_{i=1}^r \frac{\deg (q_i)}{2} q_i\frac{\partial}{\partial q_i} + \mu_A\big)(\Phi_{\mir,\bbk} ([f\Omega] )) \\
    &= \Phi_{\mir,\bbk} ([\mu_B (f\Omega )]),
\end{align*}
where the last equality follows from \cref{Chow compatible with grading}. So this implies that the bundle map is compatible with $u$-direction at $y=0$. This implies by \cref{corollary: bundle map is uniquely determined on a point} that \cref{eq:compatibility-gradings} holds for any $y$.
We deduce that $(\mir_{\bbk}, \Phi_{\mir,\bbk})$ is an isomorphism of F-bundles, concluding the proof.
\end{proof}

\subsection{\texorpdfstring{The big $D$-module mirror symmetry for $G/P$}{The big D-module mirror symmetry for G/P}} 

In this subsection, we prove the big $D$-module mirror symmetry for  $G/P$ in the framework of equivariant F-bundles.
We first construct a maximal unfolding of the small B-model equivariant F-bundle by unfolding the superpotential $\cW$ (\cref{constr:unfolded-potential}).
In particular, we discuss freeness in \cref{Lemma: formal Nakayama,B side R-linear T-structure}, and flatness in \cref{B side R-linear T-structure,remark:flatness-B-model}.
We obtain the equivariant big mirror symmetry for flag varieties in \cref{thm:big-mirror-symmetry-flag}, and deduce a non-equivariant version in \cref{thm:non-equivariant}.

\subsubsection{Unfolding of the B-model}
We fix formal variables $y = \lbrace y_1,\dots, y_N\rbrace$ and $\by = \lbrace y_{i,k} , 1\leq i\leq N, k\in\bbN^n\rbrace$.

\begin{construction}[Unfolded superpotential]\label{constr:unfolded-potential}
    For any $r<j\leq N$, let $\bar{f}_{j}\coloneqq (\Phi^{u=0,y=0}_{\rm mir})^{-1}(\sigma_{v_{j}})\in  {\rm Jac}(X^{\vee}_{P}, \mathcal{W},  p)$.
    Let $f_j\in \mathcal{O}(X^{\vee}_{P})$ be a lift of $\bar{f}_{j}$ such that the function $f_{j}$ is independent of $y_{1},...,y_{N}$.
    The unfolded superpotential $\widetilde{\cW}$ is
    $$\widetilde{\mathcal{W}}\coloneqq \mathcal{W}+\sum_{j=r+1}^N y_{j} f_j.$$
\end{construction}

Similar to the previous section on the small $B$-model equivariant $F$-bundle, we now associate to $\widetilde{\cW}$ the big $B$-model equivariant F-bundle 
\[\cF^{B,\bigquantum} \coloneqq \lbrace(\cH^{B, {\rm big}}, \nabla^{B, {\rm big}}), (\cH_R^{B, {\rm big}}, \nabla_R^{B, {\rm big}}), \id \rbrace / \bbk\dbb{\by} .\]
We first construct an $R$-linear (T)-structure $(\cH_R^{B,\bigquantum},\nabla_R^{B,\bigquantum})$ , consisting of a $R\dbb{y,u}$-module defined by
\begin{align*}
    \cH_R^{B,\bigquantum}&\coloneqq \mbox{coker} \big(R\dbb{y, u} \otimes_{\bbC[\tilde{q}_i^{\pm 1}]} \Omega^{{\rm top}-1}(X^\vee_P/Z)\overset{\partial}{\longrightarrow} R\dbb{y, u}\otimes_{{\bbC[\tilde{q}_i^{\pm 1}]}} \Omega^{\rm top}(X^\vee_P/Z)\big),\\
     \partial &\coloneqq 1\otimes \big( ud + d\widetilde{\cW}\wedge  -\sum_{j=1}^n  \lambda_j  (p^*\langle (\lambda_j)^{*}, {\rm mc}_{T^\vee}\rangle)\wedge\big),
\end{align*}
equipped with a connection defined by 
\[\nabla^{B, {\rm big}}_{R, \partial_{y_{i}}}([\omega])\coloneqq\big[\mathcal{L}_{\partial_{y_{i}}} (\omega)+ u^{-1}\frac{\partial \widetilde{\mathcal{W}}}{\partial y_{i} } \omega  - u^{-1}\sum_{j=1}^n \lambda_j \big(\iota_{\partial_{y_{i}}} p^{*} \big\langle (\lambda_j)^{*}, {\rm mc}_{T^{\vee}}  \big\rangle \big)\omega \big].\]

\begin{proposition}\label{B side R-linear T-structure}
    The data $(\cH_R^{B,\bigquantum},\nabla_R^{B,\bigquantum})/R\dbb{y}$ defines an $R$-linear (T)-structure.
\end{proposition}
\begin{proof} 
We first check that $\cH^{B,\bigquantum}_R$ is a free $R\dbb{y,u}$-module. 
By construction we have 
$\cH_R^{B,\bigquantum}\vert_{y_{\geq r+1}=0} = \cH_R^{B}$, which is a finite free $R\dbb{y_{\leq r},u}$-module (see \cite[p.\ 52]{chow2024dmodulemirrorconjectureflag}). 
Since $y_j\omega \in \im(\partial)$ if and only if $\omega \in \im(\partial)$, the element $y_j$ is torsion-free. 
Hence, the freeness of $\cH_R^{B,\bigquantum}$ follows from \cref{Lemma: formal Nakayama} below.
The flatness of $\nabla_R^{B,\bigquantum}$ follows from the facts that $\frac{\partial \widetilde{\mathcal{W}}}{\partial y_i}=f_i$ is independent of $y_1,\dots,y_N$, that  $[\cL_{\partial_{y_s}},\cL_{\partial_{y_t}}]=0$, and that $[\cL_{\partial_{y_s}},\iota_{y_t}]=0$.
\end{proof}

\begin{lemma}\label{Lemma: formal Nakayama}
Let $R_0$ be a commutative unital ring and $M$ be an $R_0\dbb{z_1,\cdots,z_m}$-module such that $z_i$ is torsion-free for all $1\leq i\leq m$. 
Let $\Omega_1,\dots, \Omega_N\in M$.
If $\{\bar \Omega_1,\cdots, \bar \Omega_N\}$ is an $R_{0}$-basis of $M/(z_{1},...,z_{m})M$, then $\{\Omega_1,\cdots,\Omega_N\}\subset M$ is
an $R_{0}\dbb{z_1,\cdots, z_m}$-basis of $M\dbb{z_1,\cdots, z_m}$.
\end{lemma}
\begin{proof}
By Nakayama lemma \cite[\S VIII.3, Corollary 2]{Zariski_commutative_II}, $\lbrace \Omega_1,\cdots,\Omega_N\rbrace$ generates $M$ as an $R_0\dbb{z_1,\cdots,\allowbreak z_m}$-module. It remains to prove the freeness. 

We first treat the case $m=1$.
Let $g_i(z_{1})\in R_{0}\dbb{z_{1}}$ be coefficients such that $\sum_i g_i(z_{1})\Omega_i=0$. 
Since $\lbrace \Omega_i\rbrace$ induces a basis of $M/z_1M$, we have $g_i(0)=0$. 
Assume $g_i$ has no terms of degree less than or equal to $b-1$ in $z_1$, i.e. we can write $g_i = z_{1}^bh_i$.
Then we have 
\begin{equation*}
    \sum_i z_{1}^b h_i\Omega_i=0.
\end{equation*}
Since $z_1^b$ is torsion-free in $M$, we deduce that $\sum_i h_i\Omega_i=0$, which implies $h_i(0)=0$.
Hence, $g_i$ has no terms of degree less than or equal to $b$.
By induction on $b$, this implies $g_i=0$.

The case $m\geq 2$ follows from the case $m=1$ by a direct induction on the number of variables.
The proof is complete.
\end{proof}

The $\bbk$-linear F-bundle $(\cH^{B,\bigquantum},\nabla^{B,\bigquantum})/\bbk\dbb{\by}$ is given by the $D$-module inverse image of $(\cH_R^{B,\bigquantum},\allowbreak \nabla_R^{B,\bigquantum})$ under the map of $R\dbb{u}$-algebras $
    R\dbb{y,u}\rightarrow R\dbb{\by,u}$ given by $y_i\mapsto \sum_{k\in \bbN^n}\lambda^k{y_{i,k}}$.
More explicitly:
\begin{align*}
    \cH^{B,\bigquantum} &\coloneqq\cH_R^{B,\bigquantum}\otimes_{\bbk\dbb{y, u}}\bbk\dbb{\by ,u} ,\\
     \nabla^{B,{\rm big}}_{\partial_{y_{i, k}}}([\omega ])&\coloneqq \big[\mathcal{L}_{\partial_{y_{i,k}}} (\omega)+ u^{-1}\frac{\partial \widetilde{\mathcal{W}}}{\partial y_{i,k} } \omega  - u^{-1}\sum_{j=1}^n \lambda_j \big(\iota_{\partial_{y_{i,k}}} p^{*} \big\langle (\lambda_j)^{*}, {\rm mc}_{T^{\vee}}  \big\rangle \big)\omega \big],
\end{align*}
equipped with the $u$-direction connection
\begin{align*}
    \nabla^{B,{\rm big}}_{u\partial_u}&\coloneqq \Gr^{B,\bigquantum} - \nabla_{E_B^{\bigquantum}}, \\
    \Gr^{B,\bigquantum} &\coloneqq u\partial_u + E_B^{\bigquantum} + \mu_B,
\end{align*}
where the Euler vector field is given by
\[E^{\rm big}_{B}\coloneqq \sum_{1\leq j\leq r}\frac{\deg(q_j)}{2}\partial_{y_{j,0}} +  \sum_{1\leq j\leq N,k\in \bbN^{n}} (1-\ell(v_j)-|k|) y_{j,k}\partial_{y_{j,k}},\]
and where the grading on differential forms $\mu_B$ (see \cref{constr:grading-fiber-small-model}) is extended to a $\bbk\dbb{\by,u}$-linear operator.
Similar to \cref{grading operator well defined.}, the total grading operator $\Gr^{B,\bigquantum}$ lifts to $\cH^{B,\bigquantum}$.

\begin{proposition}\label{remark:flatness-B-model}
The big Gauss-Manin connection $\nabla^{B, {\rm big}}$ is flat. 
\end{proposition}
\begin{proof}
 The underlying (T)-structure of $(\cH^{B,\bigquantum},\nabla^{B,\bigquantum})$ is flat, as it is obtained from the $R$-linear (T)-structure $(\cH_R^{B,\bigquantum} ,\nabla_R^{B,\bigquantum} )$, whose flatness was established in \cref{B side R-linear T-structure}.
    We check the flatness in the $u$-direction as follows.
    We have:
\begin{align*}
      \big[   \nabla_{u\partial_u}^{B, {\rm big}} , \nabla_{\partial_{y_{j,k}}}^{B,\bigquantum} ]
        &= \big[\Gr^{B,\bigquantum} , \nabla_{\partial_{y_{j,k}}}^{B,\bigquantum}\big] -\big[ \nabla_{E_B^{\bigquantum}}^{B,\bigquantum} , \nabla_{\partial_{y_{j,k}}}^{B,\bigquantum} \big]  
 \end{align*}
 Since the underlying (T)-structure of $(\cH^{B,\bigquantum} , \nabla^{B,\bigquantum})$ is flat, we have 
\begin{align}\label{eq:commutator-euler-vector-field-B-model}
    \big[ \nabla_{E_B^{\bigquantum}}^{B,\bigquantum} , \nabla_{\partial_{y_{j,k}}}^{B,\bigquantum} \big] &= \nabla_{[E_B^{\bigquantum} , \partial_{y_{j,k}}]}^{B,\bigquantum} = - (1- \ell (v_j) - \vert k\vert ) \nabla_{\partial_{y_{j,k}}}^{B,\bigquantum} .
\end{align}
We claim that the grading structure is compatible with the connection, in the sense that 
\begin{equation}\label{eq:grading-compatible-connection}
    \big[\Gr^{B,\bigquantum} , \nabla_{\partial_{y_{j,k}}}^{B,\bigquantum} \big]=-(1 - \ell(v_j) - |k|)\nabla_{\partial_{y_{j,k}}}^{B,\bigquantum}. 
\end{equation}
When $1\leq j\leq r$, the equality holds because $\nabla^{B}$ is flat.
When $r+1\leq j\leq N$, the equality holds because for any $[\omega]\in \cH^{B, big}$ we have
\begin{align*}
   \Gr^{B,\bigquantum}([f_{j}\omega])-f_{j}\Gr^{B,\bigquantum}([\omega]) &= [\ell(v_j)f_{j}\omega],  \\
   \Gr^{B,\bigquantum}([u^{-1}\omega])-u^{-1}\Gr^{B,\bigquantum}([\omega])&=-[u^{-1}\omega], \\
   \Gr^{B,\bigquantum}([\lambda^{k}\omega])-\lambda^{k} \Gr^{B,\bigquantum}([\omega])&=[|k|\lambda^{k}\omega],
\end{align*}
where we use \cref{eq: mu_b leibniz rule} and the homogeneity of the elements $f_j$, $u^{-1}$ and $\lambda^k$.
Note that $f_j$ is homogeneous of degree $\ell (v_j)$ by our choice of lift in \cref{constr:unfolded-potential}.
Now, flatness follows from equations \eqref{eq:commutator-euler-vector-field-B-model} and \eqref{eq:grading-compatible-connection}, concluding the proof.
\end{proof}

Similar to the $A$-model, we note that the quotient maps $\bbk\dbb{\by}\rightarrow \bbk\dbb{\by_{\leq r}}$ and $R\dbb{y}\rightarrow R\dbb{y_{\leq r}}$ induce an unfolding of equivariant F-bundles $\iota_B\colon\cF^B\rightarrow \cF^{B,\bigquantum}$.

\begin{proposition}\label{prop: B model is maximal}
    The morphism $\iota_B\colon \cF^{B}\rightarrow \cF^{B,\bigquantum}$ is a maximal unfolding of $\bbk$-linear equivariant F-bundles, with cyclic vector given by $[\Omega ]  :=\Phi_{\rm mir}^{-1}(1),$ where $1 \in H_T^{\ast}(X,\bbk)$ is the cyclic vector on the $A$-side.
\end{proposition}
\begin{proof}
 We only need to check that $\cF^{B,\bigquantum}$ is a maximal unfolding of $\cF^B$ with cyclic vector induced from $[\Omega ]$.
By definition of $\cF^{B,\bigquantum}$, for $r < j\leq N$ we have
\[\Phi_{\rm mir}\big([u\nabla^{B, {\rm big}}_{R, \partial_{y_j}}]\vert_{y=u=0}([\Omega] )\big)= \Phi^{u=0, y=0}_{\rm mir}(\bar{f}_j)=\sigma_{v_j}. \]
For $1\leq j\leq r$, by \cref{small mirror} we have 
\[\Phi_{\rm mir}\big([u\nabla^{B, {\rm big}}_{R, \partial_{y_j}}]\vert_{y=u=0}([\Omega] )\big) = [u\nabla_{\partial t_j}^{A} ]\vert_{t=u=0} (1) = \sigma_{v_j} .\]
Since $\Phi_{\rm mir}$ is an isomorphism and $\lbrace \sigma_{v_j} , 1 \leq j\leq N\rbrace$ is a basis of $H_T^{\ast} (X,\bbk ) $, we conclude that the unfolding is maximal.
\end{proof}

\subsubsection{Mirror symmetry}
Now we can show the big $D$-module mirror symmetry for $X=G/P$ in the following sense.
\begin{theorem}[Equivariant big mirror symmetry]\label{thm:big-mirror-symmetry-flag}
    There exists a unique isomorphism of equivariant F-bundles
    \[\big( (\mir_{\bbk}^{\bigquantum} , \Phi_{\mir,\bbk}^{\bigquantum} ), (\mir^{\bigquantum}, \Phi_{\mir}^{\bigquantum} ) \big)\colon \cF^{B,\bigquantum} \longrightarrow \cF^{A,\bigquantum}\] 
    extending the small mirror map $\big((\mir_\bbk,{\Phi_{\mir,\bbk}}), (\mir,\Phi_\mir) \big)$ in \cref{small mirror}.
\end{theorem}

\begin{proof}
By \cref{prop: A model maximal}, we have a maximal unfolding of the small $A$-model equivariant F-bundle, given by $\cF^A\rightarrow \cF^{A,\bigquantum}$. Composing the small mirror isomorphism  $\big((\mir_\bbk,{\Phi_{\mir,\bbk}}), (\mir,\Phi_\mir) \big): \cF^A\rightarrow \cF^B$ of \cref{small mirror}, with the $B$-model maximal unfolding $\iota: \cF^B\rightarrow \cF^{B,\bigquantum}$, we obtain another maximal unfolding $\cF^A\rightarrow\cF^{B,\bigquantum}$.

To apply \cref{thm:max-unfolding-arbitrary-rank} and obtain a unique isomorphism $\cF^{A,\bigquantum} \rightarrow \cF^{B,\bigquantum}$, we need to check the small $A$-model equivariant F-bundle satisfies the (GC') condition as in \cref{def:IC-GC-conditions}.
We take $v=1\in H^*_T(X,\bbk) =\cH^A|_{\tau=u=0}$. After base change to $\Frac(R)=\bbk(\lambda)$, we have the evaluation map of $\bbk(\lambda)$-modules:
\begin{align} 
    \mu_{v=1}\colon \bigoplus_{1\leq j\leq r} \bbk(\lambda)\partial_{\tau_j} &\longrightarrow H_T^*(X,\bbk)\otimes \bbk(\lambda),\\
     \partial_{\tau_j} &\longmapsto \mu (\partial_{\tau_j})(1)= \sigma_{v_j}+\lambda_{i_j} . \nonumber
\end{align}
By \cref{lemma:QHG/PH2}, $H^*_T(X,\bbk)\otimes \bbk(\lambda)$ is generated as a $\bbk(\lambda)$-algebra by $\sigma_{v_j}$ for $1\leq j\leq r$, so the orbit of $v=1$ under the action of $R[\im \mu]$ is the fiber $\cH^{A}_{R}/ (\tau_1,\dots, \tau_N,u)\cH_R^{A}\otimes_R \Frac (R)$, and (GC') is verified.
The proof is complete.
\end{proof}

We deduce a non-equivariant limit of the theorem by applying the base-change associated to the quotient map $R\rightarrow R/ (\lambda )$.
This corresponds to setting $\lambda=0$ in all the previous formulas.
We use the superscript $\lambda_0$ to indicate that the non-equivariant limit is taken.
We note that since $R/(\lambda )\simeq \bbk$, in the non-equivariant limit the equivariant F-bundles can be reduced to $\bbk$-linear F-bundles (see \cref{remark:when-R-is-k}).

On the A-side, we have 
$\tau_{i}=\sum \lambda^{k}\tau_{i, k}= \tau_{i, 0}$. 
The big quantum $D$-module $(\cH^{A,\bigquantum, \lambda_0},\nabla^{A,\bigquantum, \lambda_0})$ is an F-bundle over $\bbk\dbb{\tau} = \bbk\dbb{\tau_1,\dots, \tau_N}$ defined by
\begin{align*}
\cH^{A,\bigquantum,\lambda_0} &= H^*(X,\bbC)\otimes \bbk\dbb{\tau,u} , \\
\nabla^{A,\bigquantum, \lambda_0}_{\partial_{\tau_{j}}}&= \partial_{\tau_{j}}+u^{-1}\sigma_{v_j}\star^{\bigquantum ,\lambda_0}  ,\\
    \nabla_{u\partial_u}^{A,\bigquantum, \lambda_0} &= \Gr^{A,\bigquantum,\lambda_0} - \nabla_{E_A^{\bigquantum,\lambda_0}}^{A,\bigquantum,\lambda_0} ,
\end{align*}
where $\Gr^{A,\bigquantum,\lambda_0} = u\partial_u + E_A^{\bigquantum,\lambda_0} + \mu_A$, with the Euler vector field
\[E^{{\rm big}, \lambda_0}_{A}\coloneqq \sum_{1\leq j\leq r}\frac{\deg(q_j)}{2}\partial_{\tau_{j}} +  \sum_{1\leq j\leq N} (1-\ell(v_j)) \tau_{j}\partial_{\tau_{j}}.\]

On the B-side, we have
$y_{i}=\sum \lambda^{k}y_{i, k}= y_{i, 0}$. 
The non-equivariant big Gauss-Manin system is an F-bundle $(\cH^{B,\bigquantum, \lambda_0},\nabla^{B,\bigquantum, \lambda_0})$ over $\bbk\dbb{y} = \bbk\dbb{y_1,\dots, y_N}$ defined by
\begin{align*}
\cH^{B,\bigquantum,\lambda_0}&= \coker\big(\bbk\dbb{y, u}\otimes \Omega^{\rm top -1}(X^\vee_P/Z) \xrightarrow{1 \otimes (ud+d\widetilde{\mathcal{W}}\wedge)} \bbk\dbb{y, u}\otimes \Omega^{\rm top }(X_P^\vee/Z)\big),
\\
\nabla^{B,\bigquantum, \lambda_0}_{\partial_{y_j}}&=\cL_{\partial_{y_j}}+\frac{\partial \widetilde{\mathcal{W}}}{\partial y_j}, \\
\nabla^{B,\bigquantum, \lambda_0}_{u\partial_u}&=u\partial_u-u^{-1}\widetilde{\cW} \qquad  (\mbox{by \cref{prop:nablaBnon-equiv} 
 below}). 
\end{align*}
 \begin{proposition} \label{prop:nablaBnon-equiv}
At the non-equivariant limit, we have $\nabla^{B,\bigquantum, \lambda_0}_{u\partial_u}=u\partial_u-u^{-1}\widetilde{\cW}$.
\end{proposition}

\begin{proof}
In this proof, to simplify notations we drop the superscript $\lambda_0$.
For any $[\omega]\in \cH^{B,\bigquantum}$, we can write $\omega= g\Omega$ where  $g\in \mathcal{O}(X_{P}^{\vee})\dbb{y,u}$  and $\Omega$ is given in \cref{prop: iso_small_Dmodule current version}(2). We have
\begin{align*}
\nabla^{B , {\rm big}}_{u\partial_{u}}([g\Omega])&=[u\partial_{u}(g\Omega)+E_{B}^{\rm big}(g)\Omega+\mu_{B}(g\Omega)]-\nabla_{E_{B}^{\rm big}}^{B, {\rm big}}([g \Omega])\\
&= [u\partial_{u}(g\Omega)+\mu_{B}(g\Omega)-u^{-1}E_{B}^{\rm big}(\widetilde{\cW})g\Omega  ]
\\&= [u\partial_{u}(g\Omega)+\mu_{B}(g\Omega)-u^{-1}g\Gr^{B,\bigquantum}(\widetilde{\cW}\Omega)+ g\partial_{u}(\widetilde{\cW})\Omega+g \mu_{B}(\widetilde{\cW}\Omega) ]. 
\end{align*}
We claim that $\Gr^{B,\bigquantum} (\tcW\Omega ) = \tcW \Omega$, or equivalently that $\tcW$ is homogeneous of degree $1$ as an element of $\cO (X_P^{\vee} ) \dbb{y,u}$.
By  \cite[Lemma 4.3]{chow2025gammaconjectureiflag}, we have $\Phi_{\rm mir}([\cW\Omega])=c_1(G/P)$, hence $\cW$ is homogeneous of degree $1$ since $\Phi_{\mir}$ is graded (\cref{prop: iso_small_Dmodule current version}(5)).
For $r< j\leq N$, our choice of lift in \cref{constr:unfolded-potential} shows that $\Gr^{B,\bigquantum}(y_{j}f_{j}\Omega)=E^{\bigquantum}_{B}(y_{j})f_{j}\Omega+y_{j}\mu_{B}(f_{j}\Omega)=y_jf_{j}\Omega$.
Since $\tcW$ is independent of $u$, we have
\[ 
\nabla^{B , {\rm big}}_{u\partial_{u}}([g\Omega])= [u\partial_{u}(g\Omega)-u^{-1}\widetilde{\cW}g\Omega]+ [u \mu_{B}(g\Omega)+g\mu_{B}(\widetilde{\cW}\Omega)].  \]
By \cite[Remark 4.13]{chow2024dmodulemirrorconjectureflag} and the references therein, after fixing local coordinates $(z_i)_i$ we can write 
$\Omega=\frac{\bigwedge_{i=1}^{K} dz_{i}}{z_{1}...z_{K}}$.
By \cref{eq: mu_b leibniz rule}, we have 
\[\mu_{B}(g\Omega)=\sum_{i=1}^{K} \frac{\partial g}{\partial z_{i}} \mu_{B}(z_{i} \Omega).\]
Consider  the  $K-1$ form  
\(\eta\coloneqq\sum_{i=1}^{K} g\iota_{\partial_{z_i}} \mu_B (z_i\Omega) \).
We have:
\begin{align*}
    ud\eta &= u\sum_{i,j=1}^K \frac{\partial g}{\partial z_j} dz_j\wedge \iota_{\partial_{z_i}}\mu_B (z_i\Omega ) + u g \sum_{i=1}^K d(\iota_{\partial_{z_i}} \mu_B (z_i\Omega ) ) \\
    &= u\sum_{i=1}^K \frac{\partial g}{\partial z_i}\mu_B (z_i\Omega ) = u\mu_B (g\Omega ),
\end{align*}
and 
\begin{align*}
    d\tcW\wedge \eta &= \sum_{i=1}^K \frac{\partial \tcW}{\partial z_i} dz_i\wedge (g\iota_{\partial_{z_i}} \mu_B (z_i\Omega ))\\
    &= g\mu_B (\tcW\Omega ).
\end{align*}
We deduce that 
\[[u\mu_B (g\Omega ) + g\mu_B (\tcW \Omega )] = [ (ud+d\tcW\wedge)\eta ] = 0,\]
concluding the proof.
\end{proof}

As a direct consequence of Theorem \ref{thm:big-mirror-symmetry-flag}, we obtain the following non-equivariant big $D$-module mirror symmetry.
\begin{theorem}[Non-equivariant big mirror symmetry]\label{thm:non-equivariant}  
The non-equivariant limit of  $(\mir_{\bbk}^{\bigquantum},\Phi^{\bigquantum}_{\mir,\bbk})$ in \cref{thm:big-mirror-symmetry-flag}
gives an isomorphism of $\bbk$-linear F-bundles
\[(\cH^{A,\bigquantum, \lambda_0},\nabla^{B,\bigquantum, \lambda_0})\overset{\sim}{\longrightarrow}(\cH^{B,\bigquantum, \lambda_0},\nabla^{B,\bigquantum, \lambda_0}).\]
This isomorphism is uniquely determined by the non-equivariant small mirror isomorphism.
\end{theorem}

\begin{proof}
    The existence is clear.
    The uniqueness follows from \cite[Proposition 4.27]{HYZZ_decomposition_and_framing_of_F_bundles}, which applies because the F-bundles are maximal and admit framings.
\end{proof}

\bibliographystyle{plain}
\bibliography{dahema}

\end{document}